\documentclass[11pt]{amsart}
\usepackage{amsfonts}
\usepackage{amsmath}
\usepackage{amssymb}
\usepackage{latexsym,graphicx, epsfig}

\usepackage[usenames,dvipsnames]{color}

\usepackage{epstopdf}

\usepackage{graphicx}
\usepackage[font=sl,labelfont=bf]{caption}
\usepackage{subcaption}

\usepackage{float}
\restylefloat{table}

\usepackage{enumerate}

\newcommand{\cF}{\mathcal{F}}

\newcommand{\bE}{\mathbb{E}}
\newcommand{\bN}{\mathbb{N}}
\newcommand{\bP}{\mathbb{P}}
\newcommand{\bR}{\mathbb{R}}

\newtheorem{theorem}{Theorem}
\theoremstyle{plain}

\newtheorem{definition}[theorem]{Definition}
\newtheorem{proposition}[theorem]{Proposition}

\newtheorem{remark}[theorem]{Remark}
\numberwithin{equation}{section}
\numberwithin{theorem}{section}

\def\bth{\boldsymbol{\theta}}
\def\bx{\boldsymbol{x}}
\def\by{\boldsymbol{y}}

\begin{document}
\title[Statistics with Space-Only Noise]{Statistical Analysis of Some Evolution Equations Driven by Space-Only Noise}
\author{Igor Cialenco}
\curraddr[Igor Cialenco]{Department of Applied Mathematics, Illinois Institute of Technology\\
W 32nd Str, John T. Rettaliata Engineering Center, Room 208, \\
Chicago, IL 60616, USA}
\email[Igor Cialenco]{cialenco@iit.edu}
\urladdr{http://math.iit.edu/$\sim$igor}

\author{Hyun-Jung Kim}
\curraddr[Hyun-Jung Kim]{Department of Applied Mathematics, Illinois Institute of Technology\\
W 32nd Str, John T. Rettaliata Engineering Center, Room 208, \\
Chicago, IL 60616, USA}
\email[Hyun-Jung Kim]{hkim129@iit.edu}
\urladdr{https://sites.google.com/view/hyun-jungkim}

\author{Sergey V. Lototsky}
\curraddr[Sergey V. Lototsky]{Department of Mathematics, USC\\
Los Angeles, CA 90089, USA}
\email[Sergey V. Lototsky]{lototsky@math.usc.edu}
\urladdr{http://www-bcf.usc.edu/$\sim$lototsky}

\subjclass[2010]{Primary 62F12; Secondary 60H15, 93E10}

\keywords{Stochastic PDEs, MLE, Bayesian estimators, local asymptotic normality, regular statistical model, parabolic Anderson model, shell model, multi-channel model}

\begin{abstract}
We study the statistical properties of stochastic evolution equations driven by space-only noise, either additive or multiplicative. While forward problems, such as  existence, uniqueness, and regularity of the solution, for such equations have been studied, little is known about inverse problems for these equations. We exploit the somewhat unusual structure of the observations coming from these equations that leads to an interesting interplay between classical and non-traditional statistical models. We derive several types of estimators for the drift and/or diffusion coefficients of these equations, and prove their relevant properties.
\end{abstract}

\date{ {\small This version: \today}} %

\maketitle

\section{Introduction}
While the forward problems, existence, uniqueness, and regularity of the solution, for stochastic evolution equations have been extensively studied over the past few decades (cf. \cite{LototskyRozovsky2017Book,RozovskyRozovsky2018Book} and references therein), the  literature on statistical inference for SPDEs is, relatively speaking, limited. We refer to the recent survey \cite{Cialenco2018} for an overview of the literature and existing methodologies on statistical inference for parabolic SPDEs. In particular, little is known about the inverse problems for stochastic evolutions equations driven by \textit{space-only noise}, and the main goal of this paper is to investigate the parameter estimation problems for such equations. The somewhat unusual structure of the space-only noise exhibits interesting statistical inference problems that stay at the interface between classical and non-traditional statistical models. We consider two  classes of equations, corresponding to
two  types of noise,  \textit{additive} and  \textit{multiplicative}. As an illustration, let us take a heat equation
$$
u_t=\boldsymbol{\Delta}u, \quad t>0,
$$
on some domain and with some initial data, and where $\boldsymbol{\Delta}$ denotes the Laplacian operator.
Customarily, a random perturbation to this equation can be additive
\begin{equation}\label{eq1}
u_t=\boldsymbol{\Delta}u+\dot{W},
\end{equation}
representing a random heat source, or multiplicative
\begin{equation}\label{eq2}
u_t=u_{xx} +u\dot{W},
\end{equation}
representing  a random potential. In the case of space-dependent noise and pure point spectrum of the Laplacian $\boldsymbol{\Delta}$, one can also consider a {\em shell} version of \eqref{eq2}:
\begin{equation}
\label{eq3}
u_t=\boldsymbol{\Delta}u+\sum_k u_k \xi_k h_k(x),
\end{equation}
in which  $\{h_k,\ k\geq 1\}$ are the normalized eigenfunctions of the Laplacian, $u_k=(u,h_k),\ \  \xi_k=\dot{W}(h_k)$.
Similar decoupling of the Fourier modes is used to study nonlinear equations in fluids mechanics, both deterministic \cite[Section 8.7]{Frisch1995} and stochastic \cite{Glatt-HoltzZiane2008,FriedlanderGlatt-HoltzVicol2016};  the term ``shell model''  often appears in that context.

Our objective is to study  abstract versions of \eqref{eq3} and \eqref{eq1} with unknown coefficients:
\begin{equation}\label{eq5}
\dot{u}+\theta Au = \sigma \sum_{k=1}^{\infty} q_k u_k h_k \xi_k,
\end{equation}
and
\begin{equation}\label{eq4}
\dot{u}+\theta Au = \sigma \dot{W}^{Q},
\end{equation}
where
\begin{itemize}
	\item $A$ is a linear operator in a Hilbert space $H$;
    \item $\{h_k\}_{k\in\bN}\subset H$ are the normalized eigenfunctions of $A$ that form a complete orthonormal system in $H$, with corresponding  eigenvalues $\mu_k>0$, $k\in\bN$;
	\item $q_k>0, \ k\in\bN$, are known constants;
	\item $\theta>0$, $\sigma>0$ are unknown numbers (parameters of interest);
	\item $\xi_k, \ k\in\bN$, are independent and identically distributed (i.i.d.) standard normal random variables on the underlying probability space
	$(\Omega,\mathcal{F},\mathbb{P})$ and  $\dot{W}^{Q}=\sum_{k=1}^{\infty}q_k\xi_kh_k$;
	\item $u_k=u_k(t)=(u,h_k)_H$,  $t\in [0,T], k\in\bN$.
\end{itemize}
In each case, the solution is defined as
\begin{equation}\label{eq:FExp}
u(t)=\sum_{k=1}^{\infty}u_k(t)h_k,
\end{equation}
with
\begin{equation}\label{FK-mlt}
u_k(t)=u_k(0)\exp\Big(-(\theta \mu_k-\sigma q_k \xi_k)t\Big)
\end{equation}
for \eqref{eq5}, and
\begin{equation}
\label{FK-add}
u_k(t)=u_k(0)e^{-\theta \mu_k t}+\frac{\sigma q_k}{\theta \mu_k}\left(1-e^{-\theta \mu_k t}\right)\xi_k
\end{equation}
for \eqref{eq4}.
 For both models \eqref{eq5} and \eqref{eq4}, we assume that the observations are available in the Fourier space, namely, the observer measures the values of the Fourier modes $u_k(t)$, continuously in time for $t\in[0,T]$. In addition, for \eqref{eq4} we also consider statistical experiment when the observations are performed in physical space.
The main results of this paper are summarized as follows:
\begin{enumerate}
\item For equation \eqref{eq5}, knowledge of all $\mu_k$ is required; then, under some additional technical assumptions, the problem of joint estimation of $\theta$ and $\sigma$, using measurements in the Fourier space, leads to  statistical experiment with LAN (local asymptotic normality) and several other regularity properties.	Consequently, we prove strong consistency and asymptotic normality of maximum likelihood estimators (MLE) and Bayesian estimators for $\theta$ and $\sigma$; see Section~\ref{sec:ShellModel}.
\item For equation \eqref{eq4}, the values of $\theta\mu_k$ can be determined exactly from the observations of $u_k(t)$ at two or three time points; estimation of $\sigma$ is then reduced to estimation of variance in a normal population with known mean; see Section~\ref{sec:AdditiveFourierDomain}. Using special structure of the solution of \eqref{eq4}, and assuming zero initial conditions, and $q_k=1$, we derive consistent and asymptotically normal estimators of $\theta$ and $\sigma$, assuming measurements in the physical domain; see Section~\ref{sec:disc-samp-add}.
\end{enumerate}
In Section~\ref{sec:examples}, we present several illustrative examples, while Section~\ref{sec:Numerics} is dedicated to some numerical experiments that exemplify the theoretical results of the paper.

Throughout the paper,  given two sequences of numbers $\{a_n,\ n\geq 1\}$  and $\{b_n,\ n\geq 1\}$, we write $a_n \sim b_n$ if there exists a positive number $c$ such that $ \lim_{n\to \infty} a_n/b_n=c$.

\section{The Shell Model}\label{sec:ShellModel}

In this section we study the stochastic evolution equation \eqref{eq5}, starting with the existence and uniqueness of the solution, and continuing with parameter estimation problem for $\theta$ and $\sigma$ within the LAN framework of~\cite{IbragimovKhasminskiiBook1981}.  For better comparison with existing results, such as \cite{ZCG2018} and \cite{HuebnerRozovskii1995}, we consider a slightly more general version of \eqref{eq5}:
\begin{equation}\label{eq5g}
\dot{u}+(\theta A+A_0)u
=\sum_{k=1}^{\infty} (\sigma q_k+p_k)u_k\xi_kh_k,\ \ t>0,
\end{equation}
with known $q_k>0, \ p_k\geq 0$ and the
operators $A$ and $A_0$ such that
\begin{equation}
\label{eq:oper}
Ah_k=\mu_kh_k, \ A_0h_k=\nu_kh_k,
\end{equation}
and the real  numbers $\mu_k, \nu_k$ are known. The numbers
$\theta$ and $\sigma$ are unknown and belong to an open set
$\Theta\subseteq \bR\times (0,+\infty)$.

The solution of \eqref{eq5g} is defined by \eqref{eq:FExp}, with
\begin{equation}
\label{FK-mlt-g}
u_k(t)=u_k(0)\exp\Big(-(\theta \mu_k+\nu_k)t+(\sigma q_k+p_k) \xi_k \,t\Big).
\end{equation}

\begin{theorem}
\label{th2a}
Assume that $u(0)\in H$, $u_k(0)\not=0$, $k\geq 1$,
 and there exists a real number $C^*$ such that for all $(\theta,\sigma)\in \Theta$
 and $k\geq 1$,
 $$
 \theta\mu_k+\nu_k>C_*.
 $$

If
\begin{equation}
\label{eq:2a}
\lim_{k\to \infty} \frac{(\sigma q_k+p_k)^2}{\theta\mu_k+\nu_k}=0,
\end{equation}
for all  $(\theta,\sigma)\in\Theta$, then $u(t)\in L_2(\Omega; H)$ for all  $t>0$ and
$\bE\|u(t)\|_H^2\leq C(t,\theta,\sigma)\|u(0)\|_H^2$.

If there exist  $T>0$ and  $\bar{C}_T\geq 0$ such that
\begin{equation}\label{eq:2a-T}
 \Big( T(\sigma q_k+p_k)^2-4(\theta \mu_k+\nu_k)\Big)\leq 2\bar{C}_T,
\end{equation}
for all  $k\geq 1$ and $(\theta,\sigma)\in\Theta$,
then
 $u(t)\in L_2(\Omega; H)$ for all  $t\in [0,T]$ and
$\bE\|u(t)\|_H^2\leq \|u(0)\|_H^2\, e^{T\bar{C}_T}$.
\end{theorem}

\begin{proof}
By \eqref{FK-mlt-g},
\begin{align}
\notag
\bE u_k^2(t)&=u_k^2(0)\exp\left(-2(\theta\mu_k+\nu_k)t+
\frac{(\sigma q_k+p_k)^2t^2}{2}\right)\\
\label{FK-mlt-g-p1}
&=u_k^2(0)\exp\left(-2t(\theta\mu_k+\nu_k)
\left(1-\frac{(\sigma q_k+p_k)^2t}{4(\theta\mu_k+\nu_k)}\right)\right)\\
\label{FK-mlt-g-p2}
&=u_k^2(0)\left(\frac{t}{2}\Big((\sigma q_k+p_k)^2t-4(\theta\mu_k+\nu_k)\Big)\right).
\end{align}

If \eqref{eq:2a} holds, then, for every $t > 0$,
there exists $k=k(t)$ such that, for all $k>k(t)$,
$$
1-\frac{(\sigma q_k+p_k)^2t}{4(\theta\mu_k+\nu_k)}>\frac{1}{2},
$$
and then  \eqref{FK-mlt-g-p1} implies
$$
\bE u_k^2(t)\leq u_k^2(0)\, e^{-C_*t},\ \ k>k(t),
$$
concluding the proof.

If \eqref{eq:2a-T} holds, then \eqref{FK-mlt-g-p2}  implies that, for all $t\in [0,T]$
and $k\geq 1$,
$$
\bE u_k^2(t)\leq u_k^2(0)\, e^{T\bar{C}_T},
$$
concluding the proof.
\end{proof}

In what follows, we assume, with no loss of generality, that $C_*=0$.

Define
$$
Y_k=\frac{1}{t}\ln \frac{u_k(t)}{u_k(0)}, \quad k=1,\dots,N.
$$
Then, for each $t>0$, the random variable $Y_k$ is Gaussian with mean $-(\theta\mu_k+\nu_k)$ and variance
$(\sigma q_k+p_k)^2$, and the random variables $Y_1,\ldots, Y_N$ are independent.

 We consider $\theta$ and
$\vartheta=\sigma^2$ as the two unknown parameters.
The corresponding likelihood function becomes
\begin{equation}
\label{eq:LF}
L_N(\theta,\vartheta)=
\exp\left(
-\frac{N}{2}\ln (2\pi)-\sum_{k=1}^N\ln(\sqrt{\vartheta}\, q_k+p_k)-
\frac{1}{2}\sum_{k=1}^N
\frac{(Y_k+\theta\mu_k+\nu_k)^2}{(\sqrt{\vartheta}\, q_k+p_k)^2}\right).
\end{equation}
Direct computations produce the Fisher information matrix
\begin{align}
\label{eq:FI}
I_N&=
\left(
\begin{array}{cc}
\Psi_N(\vartheta) & 0\\
0 & \Phi_N(\vartheta)
\end{array}
\right),\ \ \ {\rm where}\\
\label{PsiPhi}
\Psi_N(\vartheta)&=\sum_{k=1}^N\frac{\mu_k^2}{(\sqrt{\vartheta}\, q_k+p_k)^2},\ \
\Phi_N(\vartheta)=\frac{1}{2}\sum_{k=1}^N
\frac{q_k^2}{( \vartheta\,q_k+\sqrt{\vartheta}\,p_k)^2}.
\end{align}
Note that if $p_k=0$ for all $k$, then $\Phi_N(\vartheta)=N/(2\sigma^4)$.
More generally, if
$$
\lim_{k\to \infty} \frac{p_k}{q_k} = c_{pq}\in [0,+\infty),
$$
then $\Phi_N\sim N$.

\begin{proposition}
\label{prop:MLE}
If $\Theta=\bR\times(0,+\infty)$ and $p_k=0$ for all $k\geq 1$,
 then the joint maximum likelihood
estimator of $(\theta,\vartheta)$ is
\begin{equation}
\label{eq:MLE}
\hat{\theta}_N=-\frac{\sum_{k=1}^N(\mu_kY_k+\mu_k\nu_k)/q_k^2}
{\sum_{k=1}^N(\mu_k/q_k)^2},\ \
\hat{\vartheta}_N=\frac{1}{N}\sum_{k=1}^N
\frac{(Y_k+\hat{\theta}_N\mu_k+\nu_k)^2}{q_k^2}.
\end{equation}
\end{proposition}

While \eqref{eq:MLE} follows by direct computation, a lot of extra work
is required to investigate the basic properties of the estimator, such as consistency and asymptotic normality, and it still will not be clear how the estimator
compares with other possible estimators, for example, Bayesian. Moreover,
when $p_k\not=0$, no closed-form expressions for $\hat{\theta}_N$ and
$\hat{\vartheta}_N$ can be found.

As a result,
 the main object of study becomes the  {\em local likelihood ratio}
\begin{equation}
\label{LocalLR}
Z_{N,\bth}(\bx)=\frac{L_N(\theta(s),\vartheta(\tau))}{L_N(\theta,\vartheta)},\ \
\bth=(\theta,\vartheta),\ \bx=(s,\tau),
\end{equation}
with
$$
 \theta(s)=\theta+\frac{s}{\sqrt{\Psi_N(\vartheta)}},\ \
 \vartheta(\tau)=\vartheta+\frac{\tau}{\sqrt{\Phi_N(\vartheta)}}.
$$
 Then various  properties of the maximum likelihood  and Bayesian estimators,
 including consistency, asymptotic normality, and optimality,
 can be established by analyzing the function $Z_{N,\bth}$; see \cite[Chapters I--III]{IbragimovKhasminskiiBook1981}.

 \begin{definition}
 	\label{def:reg}
 	The function $Z_{N,\bth}$ is called regular if the following conditions
 	are satisfied.
 	\begin{enumerate}
\item[R1.] For every compact set $K\subset\Theta$ and sequences $\bth_N=(\theta_N,\vartheta_N)$,
$\bth_N(\bx_N)=\big(\theta_N(s_N),\vartheta_N(\tau_N)\big)$  in $K$ with
$\lim_{N\to \infty}(s_N,\tau_N)=(s,\tau)$,   the representation
\begin{equation}
\label{eq:LAN}
Z_{N,\bth_N}(s_N,\tau_N)=\exp\left(
s\eta_N+\tau\zeta_N-\frac{s^2}{2}-\frac{\tau^2}{2}+\varepsilon_N(\bth_N,\bx_N)
\right),
\end{equation}
holds, so that,  as $N\to \infty$,  the random vector $(\eta_N,\zeta_N)$
converges in distribution to a standard bi-variate Gaussian vector and
the random variable $\varepsilon_N(\bth_N,\bx_N)$ converges in probability to zero.

\item[R2.] For every $\vartheta>0$,
\begin{equation}
\label{InfFI}
\lim_{N\to \infty}\Psi_N(\vartheta)=\lim_{N\to \infty}\Phi_N(\vartheta)=+\infty.
\end{equation}

\item[] To state the other  two conditions, define
$$
U_N(\bth)=\big\{(s,\tau)\in \bR^2:(\theta+s\Psi_N^{-1/2},\vartheta+\tau\Phi_N^{-1/2})\in
\Theta \big\}.
$$

\item[R3.] For every compact $K\subset \Theta$, there exist positive numbers $a$ and $B$ such that,
for all $N\geq 1$ and $R>0$,
\begin{equation}
\label{Holder}
\sup_{\bth\in K} \sup_{\scriptsize{\begin{array}{c}\bx\in U_N(\bth),\, \by\in U_N(\bth),\\ |\bx|<R,\,|\by|<R\end{array}}}|\bx-\by|^{-4}\
\bE\Big| Z^{1/8}_{N,\bth}(\bx)-Z^{1/8}_{N,\bth}(\by)
\Big|^8 \leq B(1+R^a).
\end{equation}

\item[R4.]  	For every compact set 	$K\subset \Theta$ and every $p>0$,
there exists an $N_0=N_0(K,p)$ such that
\begin{equation}
\label{PG}
\sup_{\bth\in K}\sup_{N>N_0} \sup_{\bx\in U_N(\bth)}|\bx|^p\
\bE Z^{1/2}_{N,\bth}(\bx)<\infty.
\end{equation} 		
 	\end{enumerate}
 \end{definition}

Conditions R1--R4 are  natural modifications of conditions
N1--N4 from \cite[Section III.1]{IbragimovKhasminskiiBook1981} to our setting. In particular,
 R1 is known as  uniform local asymptotic normality.  Note that, in R3, there is nothing special
about the numbers 4 and 8 except that
\begin{enumerate}
	\item  The smaller of the two numbers  should be bigger than the dimension of the
parameter space (cf. \cite[Theorem III.1.1]{IbragimovKhasminskiiBook1981});
\item In the setting \eqref{eq:LF}, \eqref{LocalLR}, the  larger
number should be at least twice as big as the smaller number, which
 is related to the square root function connecting
variance and standard deviation.
\end{enumerate}

The next result illustrates the importance of regularity.

\begin{theorem}
	\label{th:LAN}
	Assume that the function $Z_{N,\bth}$ is regular. Then
	\begin{enumerate}
		\item The joint MLE $ (\hat{\theta}_N,\hat{\vartheta}_N)$ of $(\theta,\vartheta)$
		 is consistent and  asymptotically
		normal with rate $I_N^{1/2}$, that is,
			as $N\to \infty$, $\big(\sqrt{\Psi_N(\vartheta)}(\hat{\theta}_N-\theta),
			\sqrt{\Phi_N(\vartheta)}(\hat{\vartheta}_N-\vartheta)\big)$
converges in 	
			distribution to a standard bivariate Gaussian random vector.
The estimator
		is asymptotically efficient with respect to
		loss functions of polynomial growth and,
with $\eta_N$ and $\zeta_N$ from
		\eqref{eq:LAN},
		$$
		\lim_{N\to \infty} \Big(\sqrt{\Psi_N(\vartheta)}\,(\hat{\theta}_N-\theta)-\eta_N\Big)=0,\ \
	\lim_{N\to \infty}	 \Big(\sqrt{\Phi_N(\vartheta)}\,(\hat{\vartheta}_N-\vartheta)-\zeta_N\Big)=0,
		$$
		in probability.
		\item Every Bayesian estimator $ (\tilde{\theta}_N,\tilde{\vartheta}_N)$
		 corresponding to an absolutely continuous prior on
		$\Theta$ and a loss function of polynomial growth is   consistent, asymptotically
		normal with rate $I_N^{1/2}$,  asymptotically efficient with respect to
		loss functions of polynomial growth, and
	$$
	\lim_{N\to \infty} \sqrt{\Psi_N(\vartheta)}\,\big(\hat{\theta}_N-\tilde{\theta}_N\big)=0,\ \
	\lim_{N\to \infty}	 \sqrt{\Phi_N(\vartheta)}\,\big(\hat{\vartheta}_N-\tilde{\vartheta}_N\big)=0
	$$	
in probability.
\end{enumerate}		
	
\end{theorem}

\begin{proof}
	The MLE is covered by the results of \cite[Section III.1]{IbragimovKhasminskiiBook1981}.
	The Bayesian estimators are covered by the results of \cite[Section III.2]{IbragimovKhasminskiiBook1981}.
	\end{proof}

Accordingly, our objective is to determine the conditions on the sequences
$\mu_k,\nu_k,q_k,p_k$ so that the function $Z_{N,\bth}$ defined by
\eqref{LocalLR} is regular.

\begin{theorem}
\label{th:main}
	Assume that
	\begin{align}
	\label{reg-c1}
	\sum_{k=1}^{\infty}\frac{\mu_k^2}{(q_k+p_k)^2} = +\infty,\\
	\label{reg-c2}
	\sum_{k=1}^{\infty}\frac{q_k^2}{(q_k+p_k)^2} = +\infty.
	\end{align}
	
 Then the  function $Z_{N,\bth}$ is regular.
 \end{theorem}

 \begin{proof}
 To verify condition R1 from Definition \ref{def:reg}, write
 $$
 w_{k,N}=\sqrt{\vartheta_N}\,q_k+p_k,\ \ \ \xi_{k,N}
 =\frac{Y_k+\theta_N \, \mu_k+\nu_k}{w_{k,N}}.
 $$
 Direct computations show that \eqref{eq:LAN} holds with
 \begin{equation}
 	\label{RVMain}
 \eta_N=-\frac{1}{\sqrt{\Psi_N(\vartheta_N)}}\sum_{k=1}^N \frac{\mu_k \xi_{k,N}}{w_{k,N}},\ \ \
 \zeta_N=\frac{1}{\sqrt{2\vartheta_N\,\Phi_N(\vartheta_N)}}
 \sum_{k=1}^N\frac{q_k}{w_{k,N}}\, \frac{\xi_{k,N}^2-1}{\sqrt{2}},
 \end{equation}
  and $\varepsilon_N(\bth_N,\bx_N)$ is a sum of
  $$
  \varrho_N=\frac{1}{2\Phi_N(\vartheta_N)}
  \sum_{k=1}^N \frac{\xi^2_{k,N} q_k^2}{\vartheta_N\, w_{k,N}^2} \ - \ 1
  $$
  and  several  remainder terms coming from various
    Taylor expansions. By \eqref{reg-c2},
   $\lim_{N\to \infty} \varrho_N=0$ with probability one, uniformly on
   compact subsets of $\Theta$; cf. \cite[Theorem IV.3.2]{ShiryaevBookProbability}. Convergence to zero of the
   remainder terms is routine.

Next, let  $\mathcal{H}_q$ be the $q$-the homogeneous chaos space generated by
$\{\xi_{k,N},\ k=1,\ldots, N\}$. Then equalities \eqref{RVMain} imply
$\eta_N\in \mathcal{H}_1$,
$\zeta_N\in \mathcal{H}_2$, $\bE\eta_N^2=\bE\zeta_N^2=1$,
 $\bE(\eta_N\zeta_N)=0$, and
 $$
 \lim_{N\to \infty} \bE \big(\eta_N^2+\zeta_N^2\big)^2=8,
 $$
 uniformly in $(\theta_N,\vartheta_N)$.
By  \cite[Theorem 1.1]{NourdinPeccatiSwan2014}, it follows that $(\eta_N,\zeta_N)$
converges in distribution to a standard bi-variate Gaussian vector
   and  the convergence is uniform in $(\theta_N,\vartheta_N)$. Condition R1 is now verified.

   Assumptions \eqref{reg-c1} and \eqref{reg-c2} imply R2.

  To simplify the rest of the proof,   define
 \begin{align*}
 w_{k}&=\sqrt{\vartheta}\,q_k+p_k,\ \ \
 w_{k}(\tau)=\sqrt{\vartheta(\tau)}\,q_k+p_k,\ \ \
 \xi_{k}
 =\frac{Y_k+\theta \, \mu_k+\nu_k}{w_{k}},\\
 a_k&=\frac{1}{2}\left(1-\frac{w_k^2}{w_k^2(\tau)}\right),\ \ \
 b_k=\frac{w_k\mu_k}{\sqrt{\Psi_N}\,w_k^2(\tau)}\,s,
 \end{align*}
so that
\begin{equation}
\label{ZN-proof}
Z_{N,\bth}(\bx)=\prod_{k=1}^N\left(\frac{w_k}{w_k(\tau)}\right)
\exp\left(-\frac{s^2}{2\Psi_N}\sum_{k=1}^N \frac{\mu_k^2}{w_k^2(\tau)}\right)
\exp\left(\sum_{k=1}^N\big( a_k\xi_k^2-b_k\xi_k\big)\right).
\end{equation}

To verify R3, let
$$
G_N(\bx)=Z^{1/8}_{N,\bth}(\bx).
$$
This is  a smooth function of $s$ and $w_k(\tau)$, whereas each function
 $\tau\mapsto w_k(\tau)$ is H\"{o}lder continuous of order $1/2$:
 if $N$ is small compared to $R$, then $\vartheta+\tau/\sqrt{\Phi_N}$ can
 be arbitrarily close to zero.  By the chain rule, we conclude that
 R3 hods for every fixed $N$.
It remains to verify R3 uniformly in $N$ for every fixed $K$ and $R$, and
therefore we will assume from
now on that $N$ is sufficiently large, and, in particular, $\vartheta+\tau/\sqrt{\Phi_N}$ is uniformly bounded away from zero.

By the mean value theorem,
$$
|G_N(\bx)-G_N(\by)|\leq R^{1/2}\sqrt{|\bx-\by |}\ |\nabla G_N(\bx^*)|
$$
and  $\nabla G_N(\bx)=H_N(\bx)G_N(\bx)$,
where the two-dimensional random vector
$H_N$ satisfies $\sup_N\bE|H_N|^q<\infty$ for every $q>0$.
By the H\"{o}lder inequality,
$$
\bE\big|H_N(\bx)G_N(\bx)\big|^8
\leq \Big(\bE\big|H_N|^{8(1+\varepsilon)/\varepsilon}\Big)^{\varepsilon/(1+\varepsilon)}
\Big(\bE Z_{N,\bth}^{1+\varepsilon}(\bx)\Big)^{1/(1+\varepsilon)},\
\varepsilon>0.
$$
It follows from \eqref{ZN-proof} that, for every $K$ and $R$, there is an $\varepsilon>0$ such that
$$
\sup_{\bth\in K}\sup_{\bx\in U_N(\bth),\,|\bx|<R}
\bE Z_{N,\bth}^{1+\varepsilon}(\bx)
 < B(K,R,\varepsilon),\ N>N_0(\varepsilon).
$$
Condition R3 is now verified.

To verify R4, note that, for a  standard Gaussian random variable $\xi$,
$$
\bE e^{a\xi^2-b\xi}=e^{-b^2/(4a)}\bE e^{a(\xi-(b/2a))^2}=
(1-2a)^{-1/2} e^{b^2/(2-4a)};
$$
cf. \cite[Proposition 6.2.31]{LototskyRozovsky2017Book}.
Then
$$
\bE Z_{N,\bth}^{1/2}(\bx)=\left[
\prod_{k=1}^N\left(\frac{2w_kw_k(\tau)}{w_k^2+w^2_k(\tau)}\right)^{1/2}
\right]
\exp\left(-\frac{s^2}{4\Psi_N}\sum_{k=1}^N
\frac{\mu_k^2}{w_k^2+w^2_k(\tau)}\right).
$$
To study $\big(|s|^2+|\tau|^2\big)^{p/2}\bE Z_{N,\bth}^{1/2}(\bx)$,
denote by  $C$  a number that does not depend on $N$ and $\bx=( s, \tau)$; the value of
$C$ can be different in different places.
For $p>0, \ r>0$,
$$
|s|^pe^{-rs^2}\leq \left(\frac{p}{2r}\right)^{p/2},
$$
so that
\begin{align*}
|s|^p \exp\left(-\frac{s^2}{4\Psi_N}\sum_{k=1}^N
\frac{\mu_k^2}{w_k^2+w^2_k(\tau)}\right)&\leq C
\left(\frac{\Psi_N}{\sum_{k=1}^N \mu_k^2/(w_k^2+w_k^2(\tau))}\right)^{p/2}\\
&\leq C\big(\max(1,\tau)\big)^{p/2};
\end{align*}
the last inequality follows from the definitions of $\Psi_N$ and $w_k^2(\tau)$.
Writing
$$
F_N(\tau)=|\tau|^{q}
\prod_{k=1}^N\left(\frac{2w_kw_k(\tau)}{w_k^2+w^2_k(\tau)}\right)^{1/2},
$$
the objective becomes to show that, for fixed $q>0$ and all sufficiently large $N$,
$$
\max_{\tau>-\vartheta\sqrt{\Phi_N(\vartheta)}}F_N(\tau)<\infty,
$$
which, in turn,  follows by noticing that
$$
\arg\max_{\tau>-\vartheta\sqrt{\Phi_N(\vartheta)}}
 F_N(\tau)=2\sqrt{q}+O\Big(\Phi^{-1/2}_N(\vartheta)\Big),\
N\to \infty,
$$
and
$$
\lim_{N\to \infty}F_N(2\sqrt{q})=(4q)^{q/2\,}e^{-q/2}.
$$
Condition R4 is now verified, and Theorem \ref{th:main} is proved.
  \end{proof}

Taking $\mu_k=q_k=1,\ \nu_k=p_k=0$, we recover the familiar
 problem of  joint estimation of  mean and variance in a normal population.
Because the Fisher information matrix is diagonal, violation of
one of the conditions of the
theorem still leads to a regular statistical model for the other parameter.
 For example, if \eqref{reg-c1} holds but \eqref{reg-c2} does not,
 then $\sigma$ is not identifiable, but $\theta$ is, and the local likelihood ratio
$Z_{N,\theta}(s)=Z_{N,\bth}(s,0)$  is regular, as a function of one variable.

 Conditions \eqref{eq:2a} and  \eqref{reg-c1} serve different purposes:
 \eqref{eq:2a} ensures that \eqref{eq5g} has a global-in-time solution in $H$,
  whereas \eqref{reg-c1} implies regularity of the
   estimation problem for $\theta$ based on the observations
   (multi-channel model) $u_k$, $k=1,\ldots,N,\
   N\to \infty.$
  In general, \eqref{eq:2a} and  \eqref{reg-c1} are  not related:
 with  $\theta=1,\ \mu_k=1,\ \nu_k=k^4,\ q_k=k,\ p_k=0$,
 condition \eqref{eq:2a} holds, but \eqref{reg-c1} does not;
 taking $\theta=1,\ \mu_k=k^2,\ q_k=k^{3/2},\ \nu_k=p_k=0,$ we satisfy
 \eqref{reg-c1} but not \eqref{eq:2a} [and not even \eqref{eq:2a-T}], and the
 resulting multi-channel model, while regular in statistical sense,
  does not correspond to any stochastic evolution equation.

  Condition \eqref{reg-c2} means that the numbers $p_k$ are not too big compared to
  $q_k$; for example,
  \begin{equation}
  \label{pq-sufficient}
  \limsup_{k\to \infty}\frac{p_k}{\sqrt{k}\, q_k} < +\infty
  \end{equation}
  is sufficient for \eqref{reg-c2} to hold.

 By a theorem of Kakutani \cite{Kakutani1948},
 \eqref{reg-c1} is equivalent to singularity of the measures
 \begin{equation}
 \label{prod-meas}
 \prod_{k\geq 1} \mathcal{N}\big(-(\theta\mu_k+\nu_k),(\sigma q_k+p_k)^2\big)
 \end{equation}
 on $\big(\bR^{\infty},\mathcal{B}(\bR^{\infty})\big)$
 for different values of $\theta$, and  \eqref{reg-c2} is equivalent to singularity of
 the measures \eqref{prod-meas}
 on $(\bR^{\infty},\mathcal{B}(\bR^{\infty}))$
 for different values of $\sigma$. In other words, the conditions of
 Theorem \ref{th:LAN} are in line with the general statistical paradigm that a
 consistent estimation of a parameter is  possible when, in the suitable limit,
 the measures corresponding to different values of the parameter are singular.

A similar shell model, but with space-time noise, is considered in
\cite{ZCG2018}, where the observations are
\begin{equation}
\label{shell-st}
du_k(t)+(\theta\mu_k+\nu_k)u_k(t)dt=\sigma q_ku_k(t)\,dw_k(t),\ t\in [0,T],
\end{equation}
and  $w_k=w_k(t)$ are  i.i.d. standard Brownian motions. Continuous in time observations
make it possible to determine  $\sigma q_k$ exactly from the
 quadratic variation process of $u_k$, so, with no loss of generality, we set
 $\sigma=1$. Conditions  \eqref{eq:2a} and  \eqref{reg-c1} become,
 respectively,
 \begin{equation}
 \label{eq:2a-st}
 \sup_{k\geq1} \left(\frac{q_k^2}{2}-(\theta\mu_k+\nu_k)\right)<+\infty
 \end{equation}
 and
 \begin{equation}
 \label{reg-c1-st}
 \sum_{k\geq 1} \frac{\mu_k^2}{q_k^2}=+\infty.
 \end{equation}

An earlier paper  \cite{HuebnerRozovskii1995} studies
\begin{equation}
\label{shell-st-1}
du_k(t)+(\theta\mu_k+\nu_k)u_k(t)dt= q_k\,dw_k(t),\ t\in [0,T];
\end{equation}
 now, assuming $u_k(0)=0$,
 conditions  \eqref{eq:2a} and  \eqref{reg-c1} become,
 respectively,
 \begin{equation}
 \label{eq:2a-add}
 \sum_{k\geq 1} \frac{q_k^2}{\theta\mu_k+\nu_k}<\infty
 \end{equation}
 and
 \begin{equation}
 \label{reg-c1-add}
 \sum_{k\geq 1} \frac{\mu_k^2}{(\theta\mu_k+\nu_k)^2}=+\infty.
 \end{equation}
Similar to \cite{HuebnerRozovskii1995}, set $q_k=1$ (and, in \eqref{eq5g}, also $p_k=0$),
 and assume that
the operators $\theta A+A_0$ and $A$ from \eqref{eq:oper}
 are self-adjoint elliptic of
orders $2m$ and $m_1$ respectively, in
a smooth bounded domain in $\bR^d$. It is known  \cite{Shubin} that, as $k\to +\infty$,
$$
\theta\mu_k+\nu_k\sim k^{2m/d},\ \mu_k\sim k^{m_1/d},
$$
and so
\begin{itemize}
\item conditions \eqref{eq:2a}, \eqref{reg-c1}, \eqref{eq:2a-st}, and
\eqref{reg-c1-st} always hold;
\item condition \eqref{eq:2a-add} holds if $2m>d$;
\item condition \eqref{reg-c1-add} holds if $2(m_1-m)+d\geq 0$.
\end{itemize}

More generally,  if the sequences $\{q_k,\ k\geq 1\}$ and $\{p_k,\ k\geq 1 \}$
are bounded, then \eqref{eq:2a-add} implies \eqref{eq:2a}, and \eqref{eq:2a} implies
\eqref{eq:2a-st}; whereas \eqref{reg-c1} and \eqref{reg-c1-st} are equivalent and both follow
 from \eqref{reg-c1-add}.  In other words, the space-time shell model \eqref{shell-st} admits
a global-in-time solution in $H$ and leads to a regular statistical model
under the least restrictive  conditions, and the model with additive noise
\eqref{shell-st-1} requires the most restrictive conditions.

\section{Additive Noise}\label{sec:Additive}
In this section we study the parameter estimation problem for \eqref{eq4}, driven by a space-only additive noise. We consider two observation schemes, starting with the assumption that the observations occur in the Fourier domain (similarly to shell model). Under the second observation scheme, exploring the special structure of the equation, we assume that the observer measures the derivative of the solution  in the physical space, at one fixed time point and over a uniform space grid.

Existence, uniqueness, and continuous dependence on the initial condition for equation
\eqref{eq4}  follow directly from \eqref{FK-add}.

\begin{theorem}	\label{th1}
	If $u(0)\in H$ and
	\begin{equation}\label{cond-add}
	\sum_{k=1}^{\infty}\frac{q_k^2}{\mu_k^2}<\infty,
	\end{equation}	
	then the solution of \eqref{eq4} satisfies 	$u(t)\in L_2(\Omega;H)$ for every $t>0$, and
	$$
	\bE \|u(t)\|_H^2\leq \|u(0)\|_H^2+\frac{\sigma^2}{\theta^2}
	\sum_{k=1}^{\infty}\frac{q_k^2}{\mu_k^2}.
	$$
\end{theorem}

\subsection{Observations in Fourier Domain}\label{sec:AdditiveFourierDomain}
Consider equation \eqref{eq4}. Define
$$
U_k(t)=u_k(t)-u_k(0),\ S_k(t)=1-e^{-\theta\mu_kt},\ \
F_{a,b}(x)=\frac{1-e^{-ax}}{1-e^{-bx}},\ a>b>0.
$$
The function $x\mapsto F_{a,b}(x)$ is decreasing on $(0,+\infty)$. Indeed, note that for any $p>1$, the function
$$
y\mapsto \frac{1-y^p}{1-y}
$$
is increasing  on $(0,1)$, and hence, by taking $y=e^{-bx},\ p=a/b$, the monotonicity of $F_{a,b}(\,\cdot\,)$ follows at once.

\begin{theorem}\label{th-t2}
	For every $t_2>t_1>0$ and every $k=1,2,\ldots,$
	$$
	\theta \mu_k=F^{-1}_{t_2,t_1}\left(\frac{U_k(t_2)}{U_k(t_1)}\right).
	$$
	\end{theorem}

\begin{proof}
	By \eqref{FK-add},
	\begin{equation}	\label{Uk}
	U_k(t)=\left(\frac{\sigma q_k\xi_k}{\theta\mu_k}-u_k(0)\right)S_k(t)
	\end{equation}
	and then
	$$
	\frac{U_k(t_2)}{U_k(t_1)}=F_{t_2,t_1}(\theta\mu_k),
	$$
	and since $F_{t_2,t_1}(x)$ is increasing, the inverse function $F^{-1}_{t_2,t_1}$ exists. The proof is complete.
	
\end{proof}

It turns out that making a third measurement of $U_k$ at another  specially chosen time, or by taking $t_2=2t_1$, eliminates the
need to invert the function $F_{t_2,t_1}$.

\begin{theorem}
	\label{th-t3}
	For every $t_2>t_1>0$ and every $k=1,2,\ldots,$
	$$
	\theta \mu_k=\frac{1}{t_1}\ln\frac{U_k(t_2-t_1)}{U_k(t_2)-U_k(t_1)}.
	$$ 	
	\end{theorem}
\begin{proof}
	By \eqref{Uk},
	\begin{equation*}
\begin{split}
	U_k(t_2)-U_k(t_1)&=\left(\frac{\sigma q_k\xi_k}{\theta\mu_k}-u_k(0)\right)
	\big(S_k(t_2)-S_k(t_1)\big)\\
&=
	\left(\frac{\sigma q_k\xi_k}{\theta\mu_k}-u_k(0)\right)e^{-\theta\mu_kt_1}
	S_k(t_2-t_1),
	\end{split}
\end{equation*}
	whereas
	$$
	U_k(t_2-t_1)=\left(\frac{\sigma q_k\xi_k}{\theta\mu_k}-u_k(0)\right)	 S_k(t_2-t_1).
	$$
	\end{proof}

\begin{remark}
	It is not at all surprising that the quantity $\theta\mu_k$ can be determined exactly: for every
	fixed $k$ and every collection of time moments
	$0<t_1<t_2<\cdots<t_m$, the support of the Gaussian vector
	$(U_k(t_1),\ldots, U_k(t_m))$ in $\bR^m$ is a line. As a result, the measures corresponding to
	different values of $\theta\mu_k$ are  singular, being supported on different lines. In this regard, the situation is similar to time-only noise model considered
in \cite{CialencoLototsky2009}.
\end{remark}

To estimate $\sigma$, define
$$
X_k=\frac{\theta\mu_kU_k(t)}{q_kS_k(t)}, \quad k=1,\dots, N
$$
 so that, for all $t>0$, the random variables $X_1,\ldots,X_N$ are i.i.d. Gaussian with mean
 0 and variance $\vartheta=\sigma^2$. Note that, with Theorems \ref{th-t2} and \ref{th-t3} in mind, we can indeed assume that $X_1,\ldots,X_N$ are observable.
  Then the following result is immediate.

  \begin{theorem} \label{th:AdditiveMLESigma}
  	The  maximum likelihood	estimator of $\vartheta$ is
  	$$
  	\hat{\vartheta}_N=\frac{1}{N}\sum_{k=1}^{N}
	X_k^2.
$$
This estimator has the following properties:
\begin{enumerate}
	\item It is the minimal variance unbiased estimator of $\vartheta$.
	\item It is strongly consistent:
	$\lim_{N\to \infty}\hat{\vartheta}_N=\sigma^2$ with probability one.
	\item It is asymptotically normal:
$$
\lim_{N\to \infty} \sqrt{N}(\hat{\vartheta}_N-\sigma^2)=\mathcal{N}(0,2\sigma^4)
$$
in distribution;  $\mathcal{N}(0,2\sigma^4)$ is a Gaussian random variable
with mean zero and variance $2\sigma^4$.
 \end{enumerate}
\end{theorem}

Direct computations show that
 the estimator $\hat{\vartheta}_N$ is also asymptotically efficient, both in the
Fisher sense (the lower bound in the Cramer-Rao inequality is achieved), and in the minimax sense;
for a large class of loss functions, the corresponding Bayesian estimator of $\sigma^2$ is asymptotically
equivalent to $\hat{\vartheta}_N$. For details, see \cite[Section III.3]{IbragimovKhasminskiiBook1981}.

\subsection{Observation in physical space}\label{sec:disc-samp-add}

In this section, we will consider a different sampling scheme. In contrast to the previous section, where the measurements were done in the Fourier space, here we will assume that the solution, or its spacial derivative, is observed in the physical space. It was noted in \cite{CialencoHuang2017} that,
 to estimate  the drift and/or volatility in a stochastic heat equation driven by a space-time
  white noise, it is enough to observe the solution  at only one fixed time point and at some discrete spacial points from a fixed interval. The key ingredient in the proofs was a special representation of the solution. We will follow similar arguments herein.

Let us consider the one-dimensional heat equation driven by an additive, spacial only, noise, with zero boundary conditions and zero initial data: \begin{equation}\label{eq:SimpleAdditiveModel}
\begin{cases}
\dot{u}(t,x) - \theta u_{xx}(t,x)= \sigma \dot{W}(x),\quad t>0, \ 0<x<\pi, \\
u(t,0)=u(t,\pi)=0, \ u(0,x)=0.
\end{cases}
\end{equation}
In this case, the normalized eigenfunctions of the Laplacian in $L_2((0,\pi))$ are given by $h_k =\sqrt{2/\pi} \sin(kx),\ k\in\bN$, with corresponding eigenvalues $\mu_k=-k^2, \ k\in\bN$. Moreover, we will assume that the noise is white in space, i.e. $\dot{W}(x) = \sum_{k=1}^{\infty}\xi_kh_k(x)$, where $\xi_k$ is a sequence of i.i.d. standard normal random variables on $(\Omega, \cF,\bP)$.
In view of \eqref{FK-add}, the Fourier modes of the solution of $u(t,x)$ of \eqref{eq:SimpleAdditiveModel} with respect to $\{h_k\}_{k\in\bN}$ are given by
$$
u_k(t)=\frac{\sigma}{\theta}\xi_k\frac{1-e^{-k^2\theta t}}{k^2},\quad k\in\bN.
$$
By \cite[Theorem 5.2]{KimLototsky2017},  the random field
 $u=u(t,x)$ belongs, with probability one,
 to the H\"older space $C^{3/4-\varepsilon,3/2-\varepsilon}_{t,x}((0,T)\times(0,\pi))$
for every $\varepsilon>0$. In particular, $u$ is differentiable in $x$, and
\begin{equation}\label{eq:ux}
u_x(t,x)=\sum_{k=1}^{\infty} \left(\sqrt{\frac{2}{\pi}}\frac{\sigma}
{\theta}\xi_k (1-e^{-k^2\theta t})
\frac{\cos(kx)}{k}
\right).
\end{equation}
Next, for a fixed $t>0$, we write  $u_x(t,\cdot)$ as follows:
\begin{equation}
\label{ux=BM+I}
\begin{split}
u_x(t,x)&=\frac{\sigma}{\theta} \sum_{k=1}^{\infty}
\xi_k \sqrt{\frac{2}{\pi}}\frac{\cos(kx)}{k}-
\sum_{k=1}^{\infty}\sqrt{\frac{2}{\pi}}\frac{\sigma}{\theta}\xi_k e^{-k^2\theta t}
\frac{\cos(kx)}{k}\\
&=\frac{\sigma}{\theta}\sum_{k=1}^{\infty}
\xi_k \sqrt{\frac{2}{\pi}} \left(
\frac{\cos(kx)}{k}-\frac{1}{k}\right)+
\sum_{k=1}^{\infty}\sqrt{\frac{2}{\pi}}\frac{\sigma}{\theta} \frac{\xi_k}{k}
-\sum_{k=1}^{\infty}\sqrt{\frac{2}{\pi}}\frac{\sigma}{\theta}\xi_k e^{-k^2\theta t}
\frac{\cos(kx)}{k}\\
&=:-\frac{\sigma}{\theta} B(x) + I(x).
\end{split}
\end{equation}
Clearly, for any $t>0$ and $\omega\in\Omega$, the function $I(x),  \ x\in(0,\pi)$ is infinitely differentiable.
Because  $\{h_k\}_{k\geq 1}$ is a complete orthonormal system in $L_2([0,\pi])$, the random process $B(x)=\sum_{k\geq 1}\int_0^xh_k(x)\xi_kdx$ is a standard Brownian motion on $[0,\pi]$; see for instance \cite[Section~3.1]{KlebanerBook2005}. Hence,
in view of  \cite[Proposition~2.1]{CialencoHuang2017},
for every  interval $[a,b]\subseteq [0,\pi]$, we have that
\begin{equation}\label{eq:BM+I}
V^2\left(u_x(t, \, \cdot\,);[a,b]\right)=V^2\left(\frac{\sigma}{\theta}B;[a,b]\right)=
\frac{\sigma^2}{\theta^2} (b-a),
\end{equation}
where $V^2(Y; [a,b])$ denotes the quadratic variation of process $Y$ on interval $[a,b]$, and over uniform partition, i.e.
$$
V^2(Y; [a,b]):= \lim_{M\to\infty} \sum_{k=0}^{M-1} |Y(x_{k+1})-Y(x_k)|^2
$$
with probability one, and where $x_k=a+(b-a)k/M, \ k=0,\ldots,M$. Assume that, for some fixed $t>0$ we measure $u_x(t,x)$ at the grid points
\begin{equation}\label{eq:grid}
\{(t,x_j),\ j=0,...,M\},
\end{equation}
where $a=x_0<x_1<\cdots<x_M=b$. Using the definition of the quadratic variation, we take the following natural estimates of $\theta^2$ and $\sigma^2$ within this sampling scheme:
\begin{align}
\check{\theta}^2_{M}:=\frac{\sigma^2(b-a)}{\sum_{j=1}^M \left(u_x(t,x_j)-u_x(t,x_{j-1})\right)^2}, \label{eq:thetaEST}\\
\check{\sigma}^2_{M}:=\frac{\theta^2 \sum_{j=1}^M \left(u_x(t,x_j)-u_x(t,x_{j-1})
\right)^2}{b-a}. \label{eq:sigmaEST}
\end{align}
As next result shows, both estimators are strongly consistent and asymptotically normal.

\begin{theorem}\label{th:thetaC-AN}
If  $\sigma$ is known, then  $\check{\theta}^2_{M}$, as an estimator of $\theta^2$,
 is strongly consistent:
$$
\lim_{M\rightarrow \infty} \check{\theta}^2_{M}=\theta^2,\
\mathbb{P}-a.s.,
$$
and asymptotically normal:
$$
\lim_{M\to \infty} \sqrt{M}\left(\check{\theta}^2_{M}-\theta^2\right)
=\mathcal{N}(0,2\theta^4)\quad \mbox{in distribution}.
$$

If  $\theta$ is known, then, $\check{\sigma}^2_{M}$, as an estimator of $\sigma^2$,
 is  strongly consistent:
$$
\lim_{M\rightarrow \infty} \check{\sigma}^2_{M}=\sigma^2,\
\mathbb{P}-a.s.,
$$
and asymptotically normal:
$$
\lim_{M\to \infty} \sqrt{M}\left(\check{\sigma}^2_{M}-\sigma^2
\right) = \mathcal{N}(0,2\sigma^4)\quad \mbox{in distribution}.
$$
\end{theorem}
The proof is a direct consequence of \cite[Theorem 3.1 and 3.2]{CialencoHuang2017}.

\medskip

In reality, the observer usually has a direct access to $u$ rather than $u_x$.
It is therefore natural  to replace the values of $u_x(t,x)$ in \eqref{eq:thetaEST} and \eqref{eq:sigmaEST} by their finite difference approximations, for example using the forward finite difference $(u(t,x+\delta) - u(t,x))/\delta$, with $\delta=(b-a)/M$, and consider the following estimators for $\theta^2$ and $\sigma^2$:
\begin{align}\label{eq:estTheta2}
\tilde{\theta}^2_{M}&:= \frac{\sigma^2(b-a)^3}{M^2\sum_{j=1}^{M-1}\left(u(t,x_{j+1})-2u(t,x_j)+u(t,x_{j-1})\right)^2}, \\
\label{eq:estSigma2}
\tilde{\sigma}^2_{M}&:= \frac{\theta^2M^2\sum_{j=1}^{M-1}\left(u(t,x_{j+1})-2u(t,x_j)+u(t,x_{j-1})\right)^2}{(b-a)^3}.
\end{align}
Note that $u(t,\cdot)$ is H\"older continuous of order $3/2-\varepsilon$, for any $\varepsilon>0$, and  higher order finite difference approximations are not immediately applicable.  We conjecture that these estimators are also consistent and asymptotically normal, while the rigourous proof of asymptotic properties of these estimators remain an open problem. It is also interesting to note that naive numerical methods of approximation of the solution lead to undesirable results; see Example~2 for more details.

\section{Examples}\label{sec:examples}

In this section, we will present several examples of SPDEs that fit the theoretical results derived in previous sections.

Let $G$ be a bounded and smooth domain in $\mathbb{R}^d$, and let us consider the Laplace operator $\boldsymbol{\Delta}$
on $G$ with zero boundary conditions. It is well known  \cite{Shubin} that $\boldsymbol{\Delta}$ has only point spectrum,
 the set of normalized eigenfunctions is a complete orthonormal system in $H=L_2(G)$, and,
with  $\lambda_k, \ k\in\mathbb{N}$, denoting the eigenvalues of $-\boldsymbol{\Delta}$, arranged in increasing order, $\lambda_k \sim  k^{2/d}$.

We take $A = (-\boldsymbol\Delta)^\beta$, and $A_0 = (-\boldsymbol\Delta)^{\beta_0}$, for some $\beta, \ \beta_0>0$.
Then
$$
\mu_k \sim k^{2\beta/d}, \quad \nu_k \sim k^{2\beta_0/d}.
$$

\smallskip\noindent
\textbf{Shell Model.}
We consider the following equation
\begin{equation}\label{eq:FracLap}
\dot{u}+\left(\theta (-\boldsymbol\Delta)^\beta+ (-\boldsymbol\Delta)^{\beta_0}\right)u =\sum_{k=1}^{\infty} (\sigma q_k+p_k)u_k\xi_kh_k,\ \ t>0,
\end{equation}
with $u(0)\in H$, $\sigma>0$,  and $\bar{\beta}=\max\left(\beta, \beta_0\right)>0$, so that
$$
\mu_k\sim k^{2\beta/d},\quad \nu_k\sim k^{2\beta_0/d},\quad
\theta\mu_k+\nu_k\sim k^{2\bar{\beta}/d};
$$
when $\beta\geq \beta_0$, the last relation also imposes a  condition on $\theta$ in  the form
of a lower bound  $\theta>\theta_0$ for some $\theta_0\in \mathbb{R}$.

If
\begin{equation}
\label{pq-laplace}
q_k+p_k = o\big(k^{\bar{\beta}/d}\big),
\end{equation}
then \eqref{eq:2a} holds and
 \eqref{eq:FracLap} is well posed on $[0,T]$ for every $T$.

 If
 \begin{equation}
 \label{pq-laplace-1}
 q_k+p_k = O\big(k^{\bar{\beta}/d}\big),
 \end{equation}
 then \eqref{eq:2a-T} holds and
 \eqref{eq:FracLap} is well posed on $[0,T]$ for sufficiently small $T$.

To proceed, let us first assume that $q_k=1$ and $p_k=0$. Then \eqref{pq-laplace}
and \eqref{reg-c2} hold, whereas  \eqref{reg-c1} becomes
 \begin{equation}
 \label{ex-c1-1}
  \beta\geq -\frac{d}{4};
 \end{equation}
 with a strict inequality in \eqref{ex-c1-1}, we get
 $$
 \Psi_N\sim N^{(4\beta+d)/d},\ \Phi_N\sim N.
 $$

More generally, if $q_k+p_k\sim k^r$, $0\leq r<\bar{\beta}/d$,  and
\eqref{pq-sufficient} holds,  then  \eqref{pq-laplace}
and \eqref{reg-c2} hold, whereas  \eqref{reg-c1} becomes
$$
\beta\geq \frac{rd}{2}-\frac{d}{4},
$$
 In the ``critical'' case $r=\bar{\beta}/d$ (cf. \eqref{pq-laplace-1}), we get
 $$
 \beta\geq \frac{\bar{\beta}}{2}-\frac{d}{4},
 $$
 which is similar to the corresponding condition from \cite{HuebnerRozovskii1995}.

 On the other hand, if $q_k+p_k\sim e^{-k}$, then no additional conditions on $\beta$
 are necessary to satisfy \eqref{reg-c1}; for example,  if $\beta_0>0$, then both
 \eqref{eq:2a} and \eqref{reg-c1} hold for every $\beta\in \bR$.

\smallskip\noindent
\textbf{Additive Model.} We now consider the fractional heat equation driven by  additive noise
\begin{equation}\label{eq:FrLapAdd}
  \dot{u}+\theta (-\boldsymbol\Delta)^{\beta}u = \sigma \sum_{k=1}^{\infty}q_k\xi_kh_k, \ t>0,
\end{equation}
with $u(0)\in H$ and $\beta\in\mathbb{R}$.
The existence and uniqueness of the solution, and all asymptotic properties of the considered estimators hold true if \eqref{cond-add} is satisfied, which now becomes
$$
\sum_{k=1}^{\infty}q_k^2 k^{-4\beta/d} <\infty.
$$
In particular, one can take
$$
q_k\sim k^{\delta} \big(\ln k\big)^{r},\ \delta<\frac{2\beta}{d}-\frac{1}{2},\ r\in \mathbb{R}.
$$
Note that if   $\beta\leq 0$,  then equation \eqref{eq:FrLapAdd}, while not an SPDE,
 can still  be  a  legitimate stochastic evolution
equation.

\section{Numerical Experiments}\label{sec:Numerics}

\noindent
\textbf{Example 1. Shell model.} Let us consider the equation \eqref{eq:FracLap} in dimension $d=1$, and $G=[0,\pi]$. Hence, $\lambda_k=k^2, \ h_k(x)=\sqrt{2/\pi}\sin(kx), \ k\in\bN$. We take the following set of parameters
$$
\beta=1,\ \beta_0= 0.5, \  \theta_0=0.5, \ \sigma_0=0.6, \ T=1, \ q_k=p_k=k,\ u(0)= x(\pi-x).
$$
Using this set of parameters, we simulated $M=5,000$ paths of the first $60$ Fourier coefficients \eqref{FK-mlt} of the solution $u(t,x)$ on a fine time grid $\delta t= 0.01$;
note that implementation of  \eqref{FK-mlt} requires no numerical approximation. Using Theorem~\ref{th:LAN} we compute the MLEs for  $\hat\theta_N$ and $\hat\sigma_N$ for each path, and consequently their sample mean and sample standard deviation. In Figure~\ref{fig:Ex1ShellFig1}, we present one realization of the estimators (circled lines) $\hat\theta_N$ and $\hat\sigma_N:=(\hat\vartheta_N)^{1/2}$, as well as the true values of the parameters (solid lines).  In Figure~\ref{fig:Ex1ShellFig2}, we display the sample mean of $\hat{\theta}_N$ and $\hat\sigma_N$.
 As expected, the estimates and their sample means converge to the true value of the parameters of interest, as the number of Fourier modes $N$ increases. Moreover, as displayed in Figure~\ref{fig:Ex1ShellFig3}, the rate of convergence of the sample standard deviation coincides with the theoretical rate given by the asymptotic normality result. Finally, in Figure~\ref{fig:Ex1ShellFig4} (left panel) we present the empirical distribution of $\hat\theta_N-\theta$ for $N=60$, superposed on the distribution of a Gaussian random variable (solid line) with mean zero and variance $1/\Psi_N$. We also present the Q-Q~plot of these two distributions; Figure~\ref{fig:Ex1ShellFig5} (left panel). The right panels of Figure~\ref{fig:Ex1ShellFig4} and Figure~\ref{fig:Ex1ShellFig5} contain similar plots for $\hat\sigma_N$. Figures~\ref{fig:Ex1ShellFig4}~and~\ref{fig:Ex1ShellFig5} validate the asymptotic normality of these estimators. In conclusion, the obtained numerical results are consistent with the theoretical results from Theorem~\ref{th:LAN}.

\begin{figure}[!ht]
    \centering
    \begin{subfigure}[t]{0.47\textwidth}
        \centering
        \includegraphics[width=\linewidth]{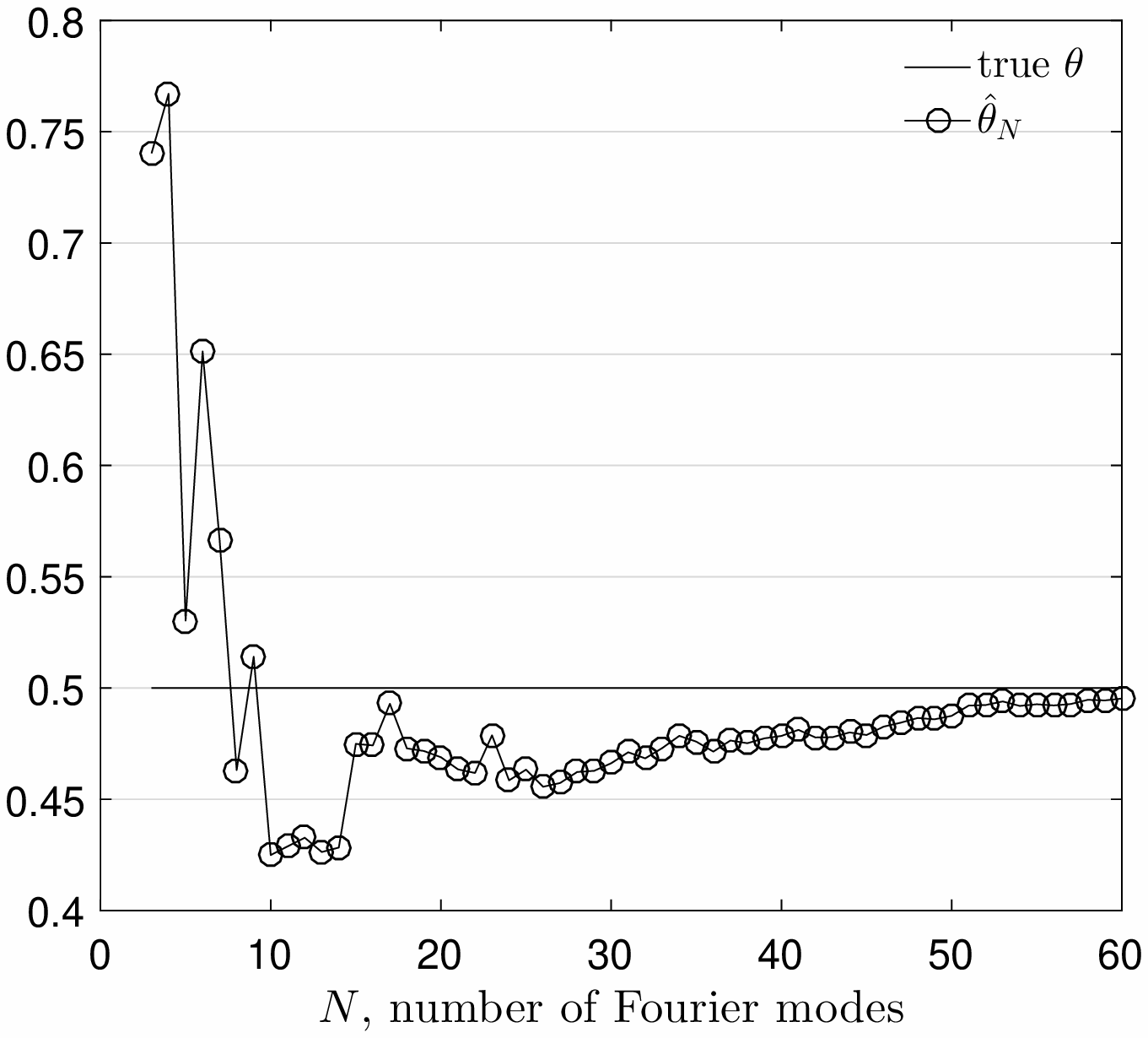}
    \end{subfigure}
    \hfill
    \begin{subfigure}[t]{0.52\textwidth}
        \centering
        \includegraphics[width=\linewidth]{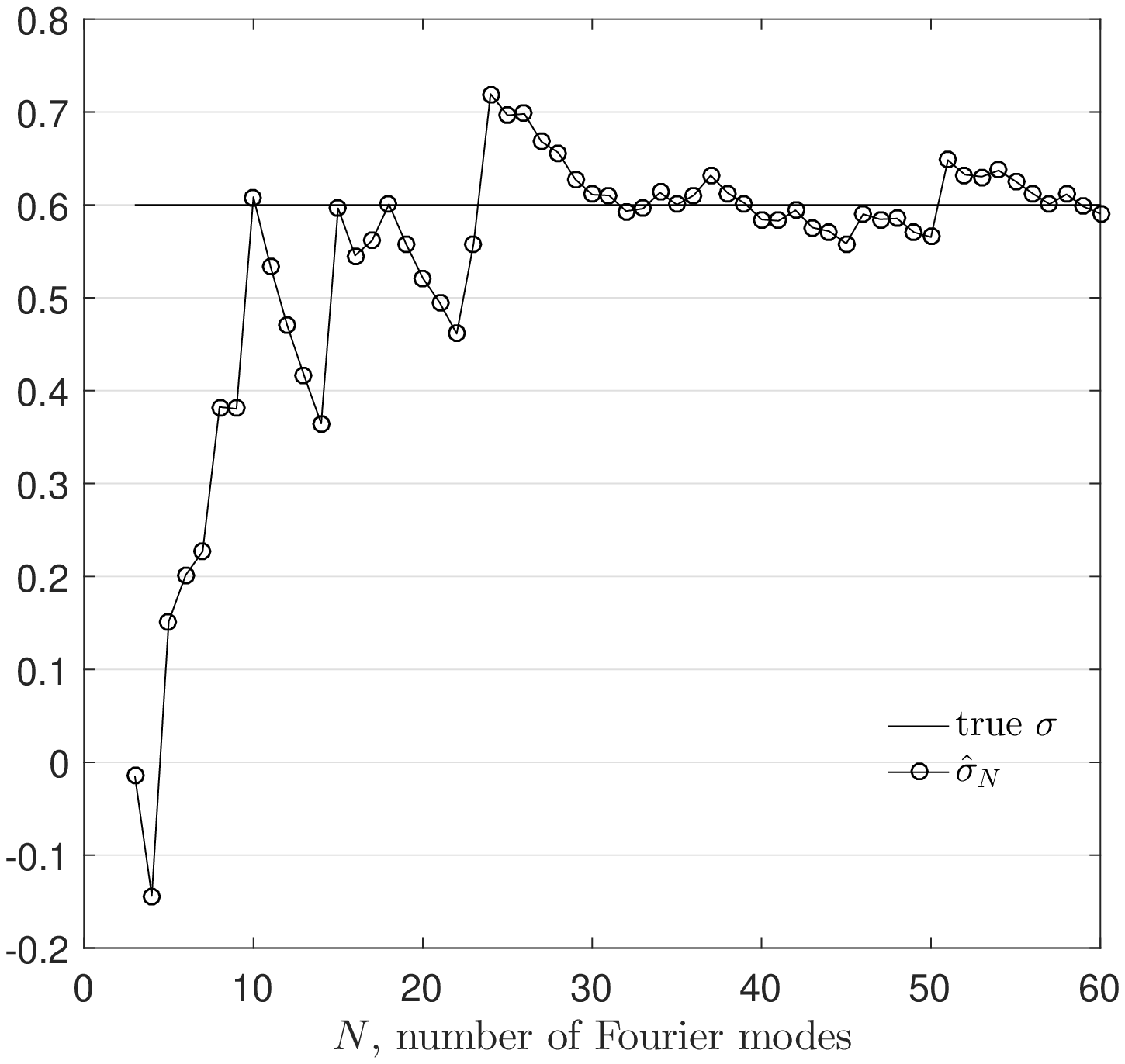}
    \end{subfigure}
    \caption{Shell model. Sample path of $\hat\theta_N$ (left panel) and $\hat\sigma_N$ (right panel).}
    \label{fig:Ex1ShellFig1}
\end{figure}

\begin{figure}[!ht]
    \centering
    \begin{subfigure}[t]{0.49\textwidth}
        \centering
        \includegraphics[width=\linewidth]{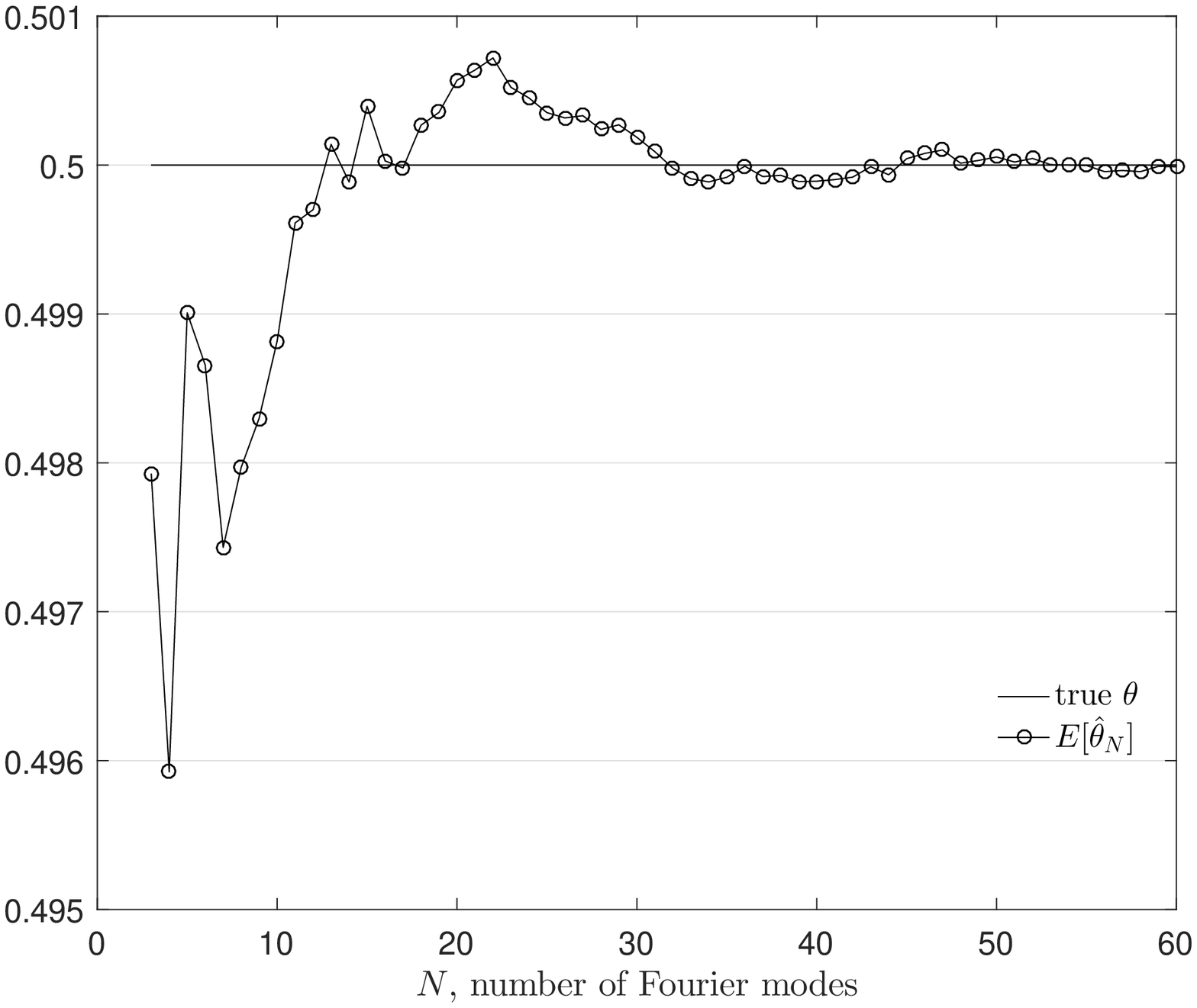}
    \end{subfigure}
    \hfill
    \begin{subfigure}[t]{0.49\textwidth}
        \centering
        \includegraphics[width=\linewidth]{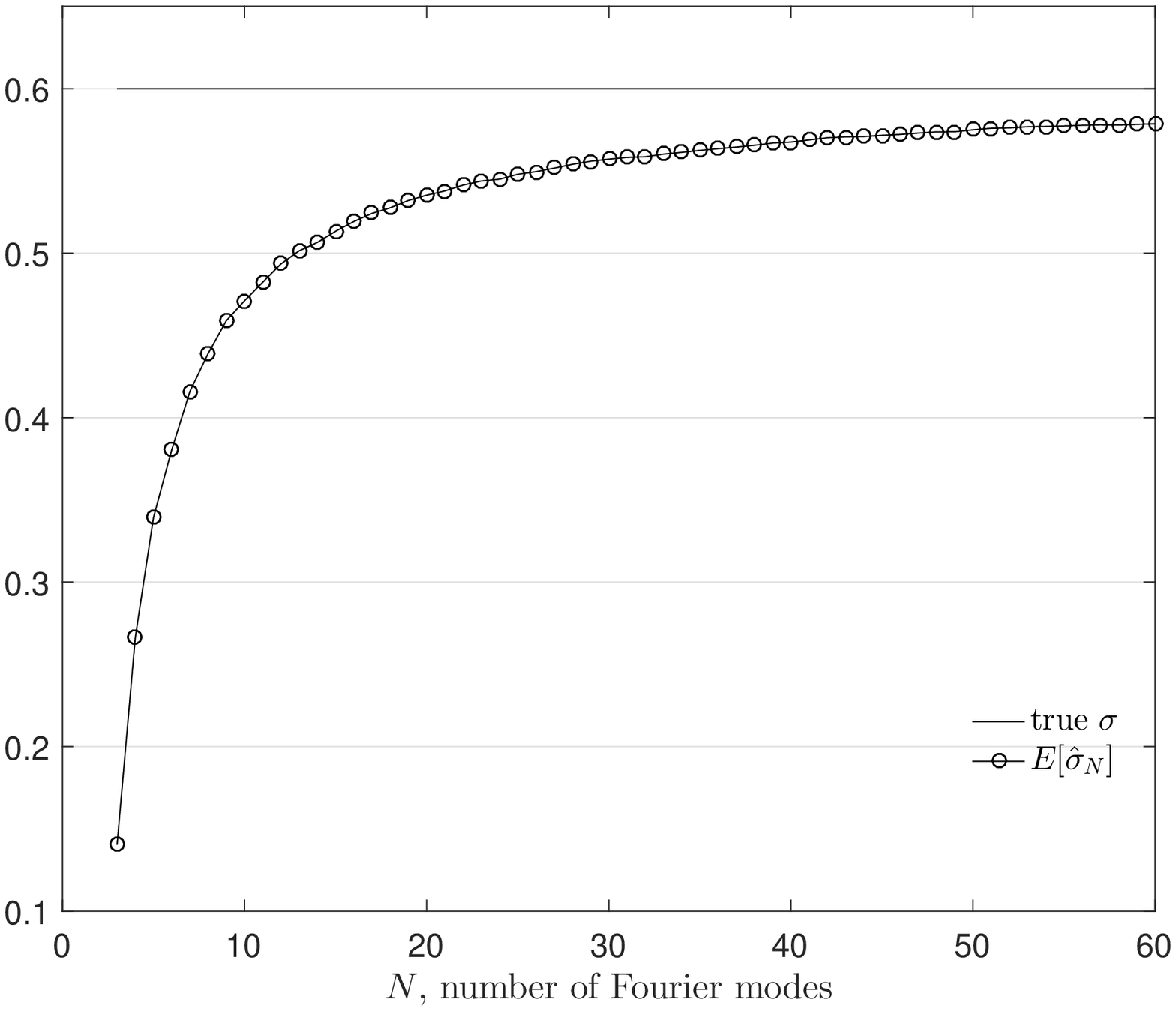}
    \end{subfigure}
    \caption{Shell model. Sample mean of $\hat\theta_N$ (left panel) and $\hat\sigma_N$ (right panel).}
    \label{fig:Ex1ShellFig2}
\end{figure}

\begin{figure}[!ht]
    \centering
    \begin{subfigure}[t]{0.48\textwidth}
        \centering
        \includegraphics[width=\linewidth]{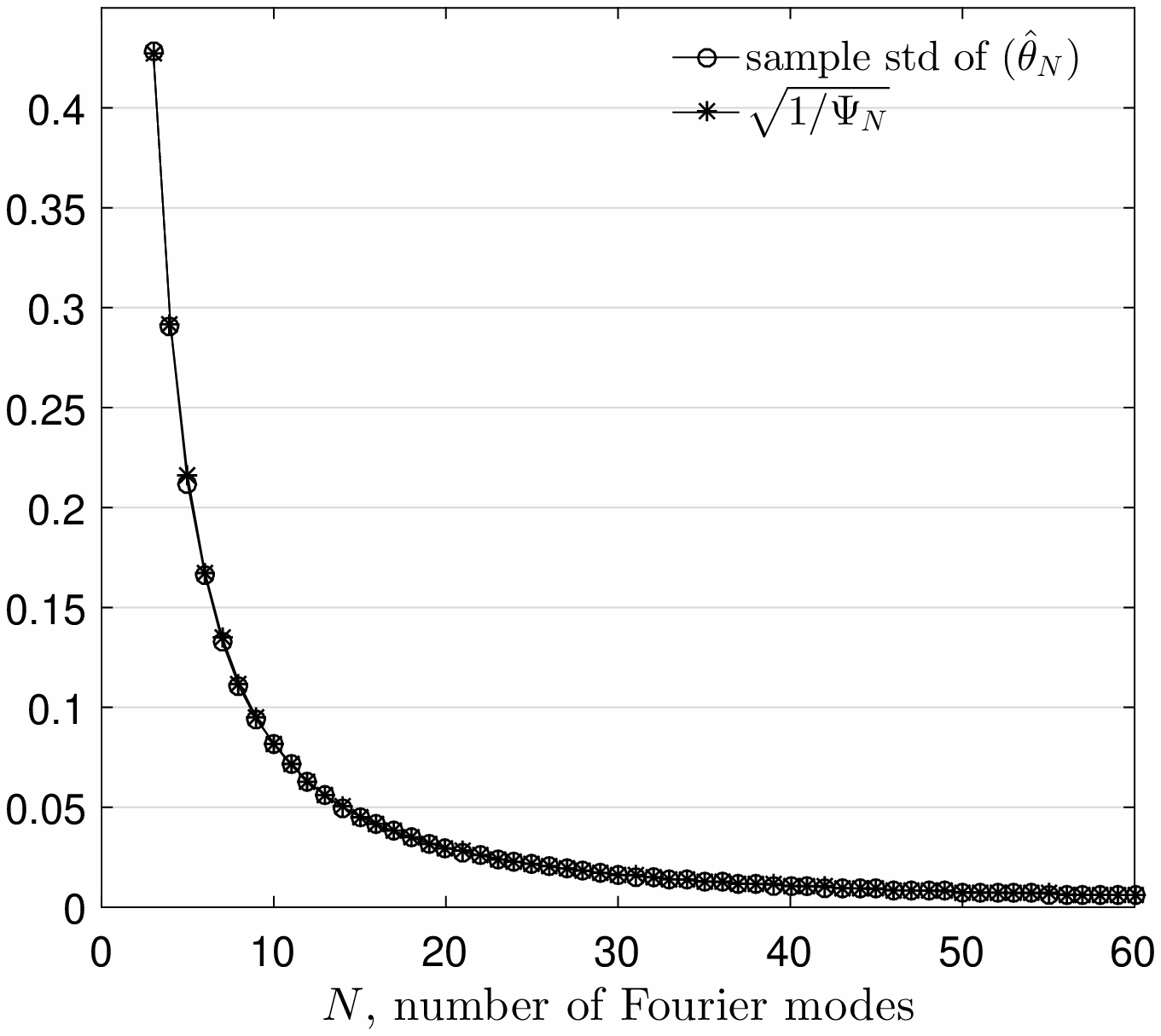}
    \end{subfigure}
    \hfill
    \begin{subfigure}[t]{0.48\textwidth}
        \centering
        \includegraphics[width=\linewidth]{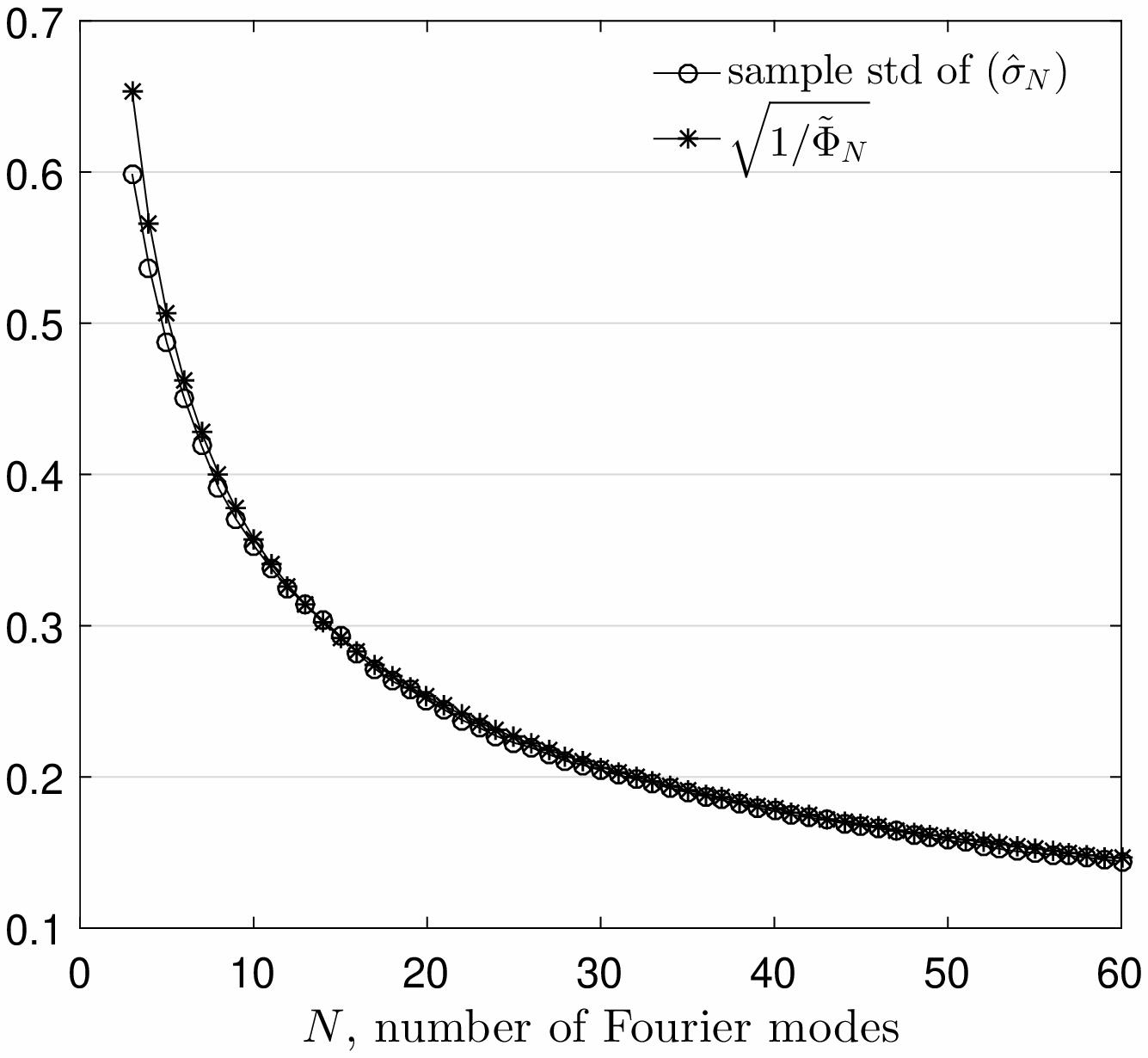}
    \end{subfigure}
    \caption{Shell model. Left panel: Sample standard deviation of $\hat\theta_N$ and theoretical standard deviation $\sqrt{1/\Psi_N}$ from asymptotic normality;
     Right panel: Sample standard deviation of $\hat\sigma_N$ and theoretical standard deviation $\sqrt{1/\widetilde\Phi_N}$, where $\widetilde\Phi_N$ is the Fisher information for $\hat\sigma_N$.}
    \label{fig:Ex1ShellFig3}
\end{figure}

\begin{figure}[!ht]
    \centering
    \begin{subfigure}[t]{0.48\textwidth}
        \centering
        \includegraphics[width=\linewidth]{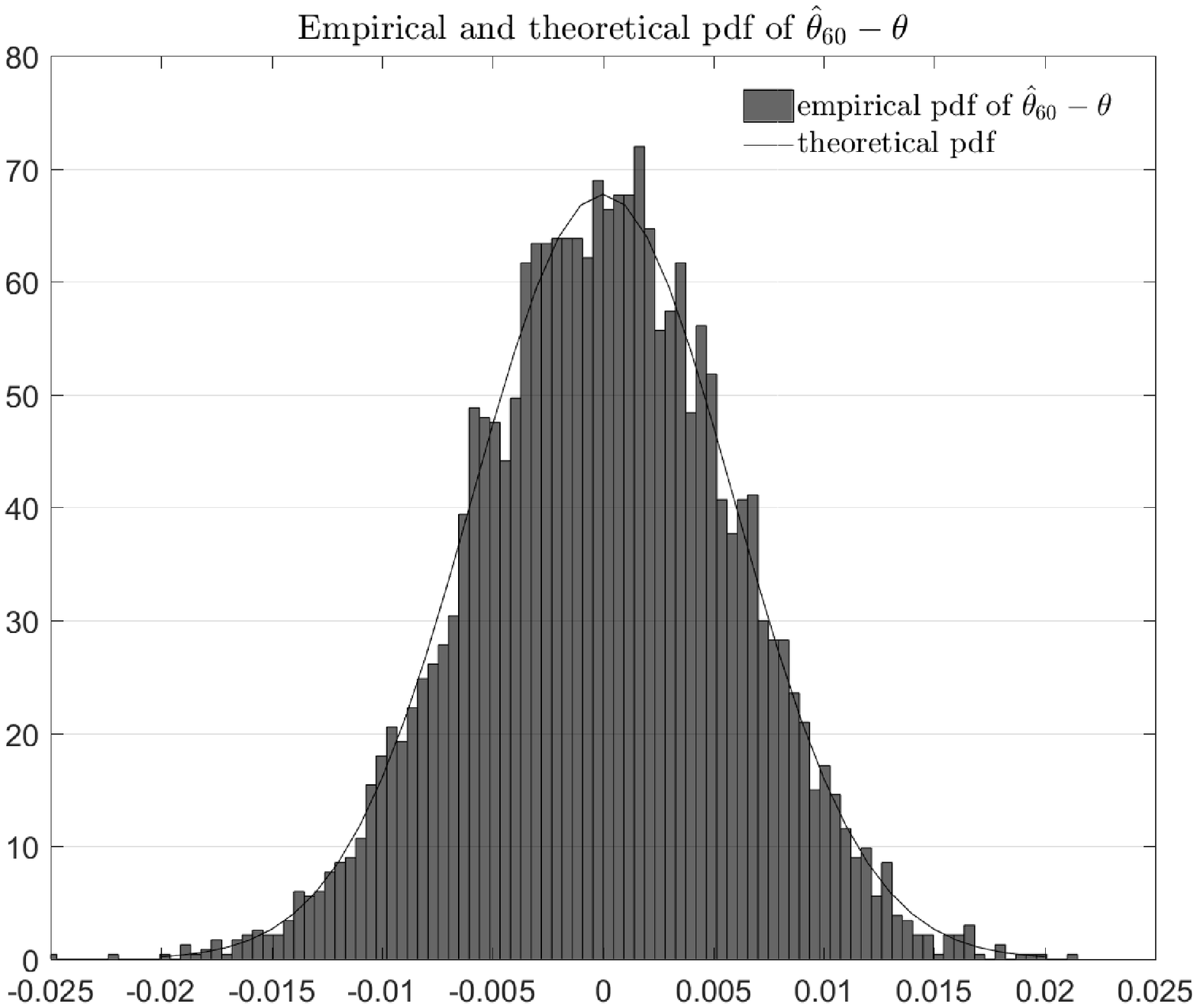}
    \end{subfigure}
    \hfill
    \begin{subfigure}[t]{0.48\textwidth}
        \centering
        \includegraphics[width=\linewidth]{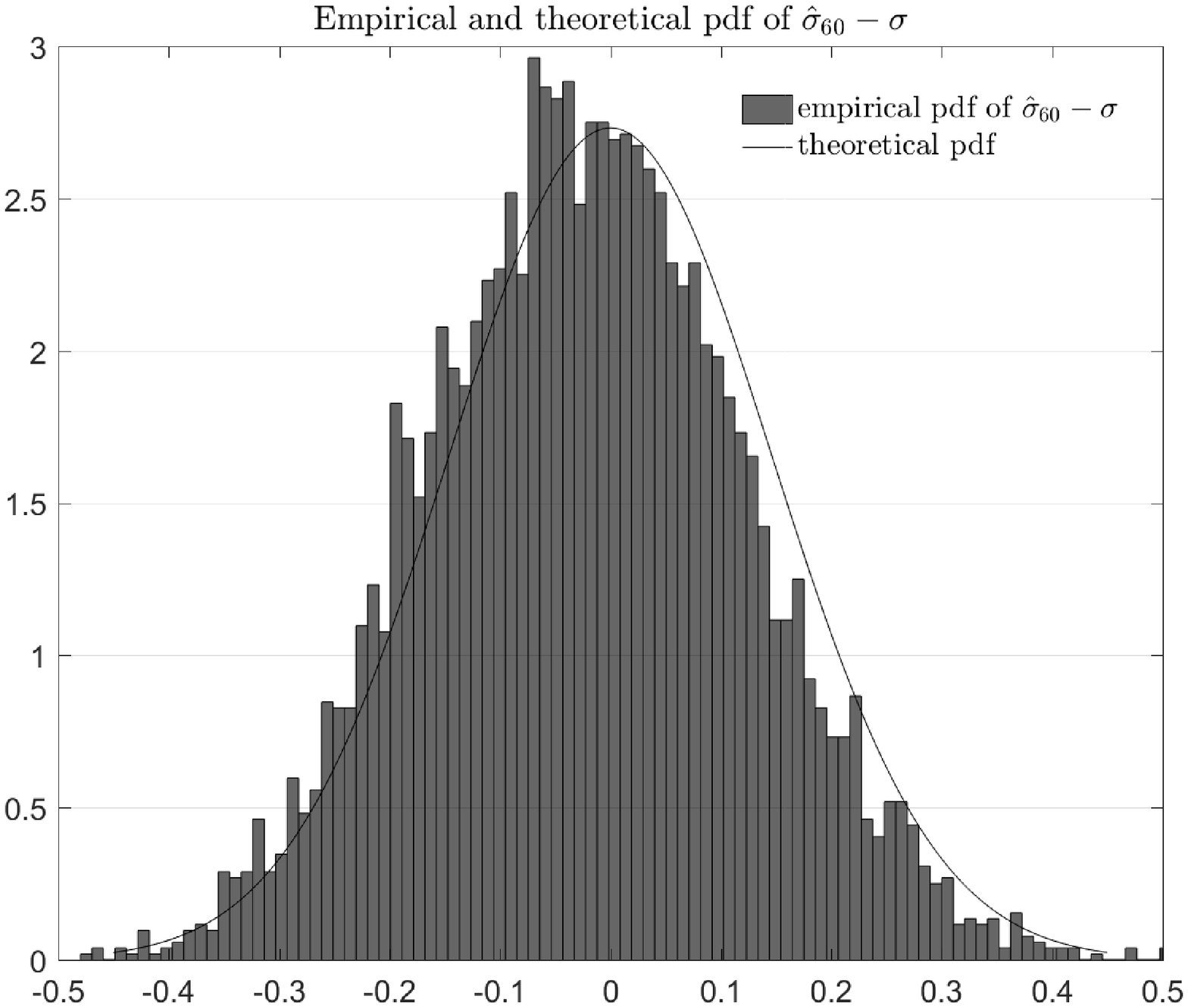}
    \end{subfigure}
    \caption{Shell model. Empirical distribution of $\hat\theta_{60}-\theta$ (left panel) and $\hat\sigma_{60}-\sigma$ and the pdf (solid lines) of the theoretical normal distribution from asymptotic normality.}
    \label{fig:Ex1ShellFig4}
\end{figure}

\begin{figure}[!ht]
    \centering
    \begin{subfigure}[t]{0.48\textwidth}
        \centering
        \includegraphics[width=\linewidth]{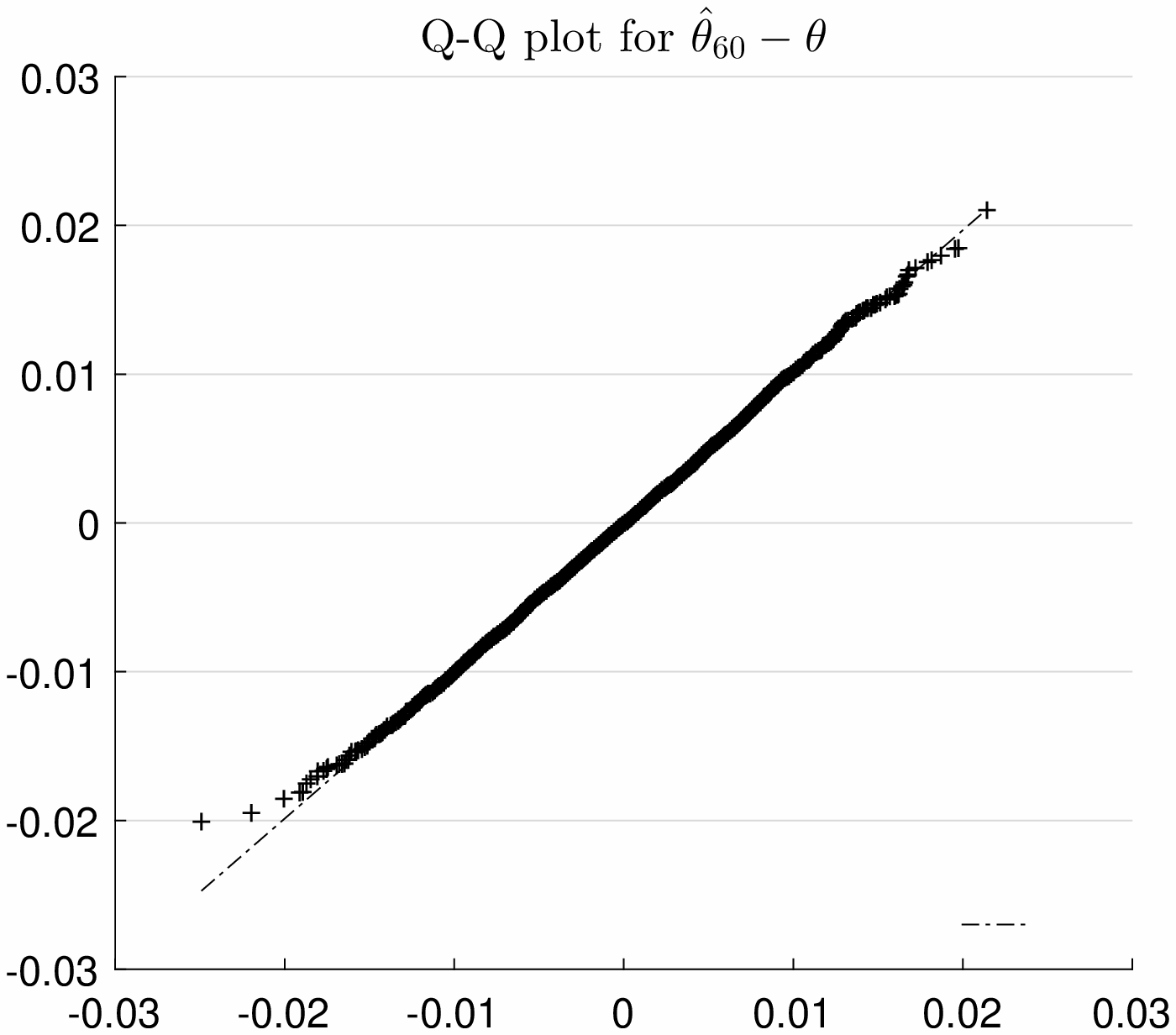}
    \end{subfigure}
    \hfill
    \begin{subfigure}[t]{0.48\textwidth}
        \centering
        \includegraphics[width=\linewidth]{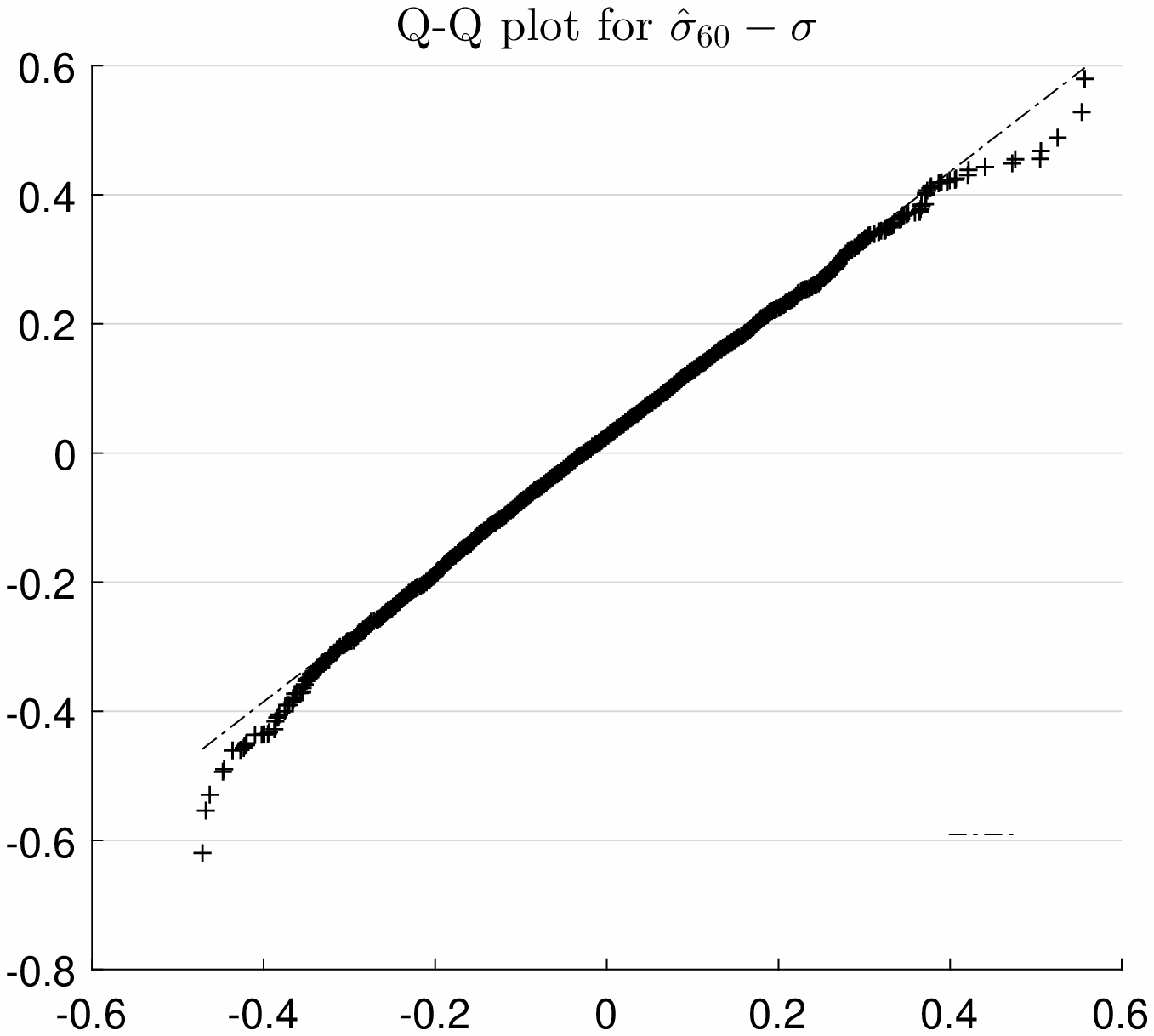}
    \end{subfigure}
    \caption{Shell model. Q-Q-Plot of $\hat\theta_{60}-\theta$ (left panel) and $\hat\sigma_{60}-\sigma$   vs the corresponding theoretical normal distribution from asymptotic normality.}
    \label{fig:Ex1ShellFig5}
\end{figure}

\begin{remark}
We ran the numerical experiments for different shell models of the type  \eqref{eq:FracLap}, and all obtained results agree with theoretical ones. For example, with $\mu_k=1, \ \nu_k=k^4, \ q_k=k, \ p_k=0$ and all other parameters as in Example~2, the solution exists, but \eqref{reg-c1} is not satisfied, and as expected, estimators do not converge. On the other hand, with $\mu_k=k^2,\ q_k=k^{3/2},\ \nu_k=p_k=0$, solution of \eqref{eq:FracLap} does not exists, but \eqref{reg-c1} and \eqref{reg-c2} are satisfied, and formally computed estimates converge.
Other set of parameters, e.g. $\mu_k=k^2, \ \nu_k=0, \ q_k= k^{\frac14}(\log(1+k))^{\frac12}, \ p_k=0$,  for which the solution exists and \eqref{reg-c1} and \eqref{reg-c2} are satisfied, produce similar results as in Example~1. We also computed the estimates for $\theta$ and $\sigma$ using Bayesian approach, and overall the results look similar to the MLE, although they are less stable numerically and more advanced numerical methods may need to be implemented.  \end{remark}

\noindent
\textbf{Example 2. Additive noise.} We consider the equation \eqref{eq:FrLapAdd}, with $\beta=1$, $G=[0,\pi]$, $d=1$.
Thus, $\lambda_k=k^2, \ h_k(x)=\sqrt{2/\pi}\sin(kx), \ k\in\bN$, and we take the following set of parameters
$$
\theta_0=0.1, \ \sigma_0=0.1, \ T=1, \ q_k=1, u(0)=0.
$$
We will use similar numerical experiments as in Example~1, and compute the Fourier modes $u_k$ by applying directly \eqref{FK-add}.  First we assume that $\sigma$ is known and apply Theorem~\ref{th-t2} and, respectively,  Theorem~\ref{th-t3} to compute `the exact estimators' for $\theta$ by using the values of the Fourier modes at two time points and, respectively, three time points, not counting the value at $t=0$. The obtained estimated value $\theta'$ for $\theta$ are virtually indistinguishable from the true parameter, with the relative error $|\theta'-\theta|/\theta<10^{-12}$, for any combination of chosen time points in $(0,1)$, and/or the Fourier mode $u_k, k=1,\ldots,60$.

Next we assume that $\theta$ is known, and we estimate $\sigma$ by the MLE from Theorem~\ref{th:AdditiveMLESigma}.  We conducted several numerical experiments to confirm the results of Theorem 3.5, and all experiments produced results similar to what we present below. In Figure~\ref{fig:Ex1AdditiveFig2} (left panel) we display one typical realization of $\hat\sigma_N$ (circled lines), which converges to the true value $\sigma=0.1$ (solid line).
Using $5,\!000$ simulated paths of the first 55 Fourier coefficients, we compute the sample mean of $\hat\sigma_N$, presented in Figure~\ref{fig:Ex1AdditiveFig2} (right panel). Sample standard deviation of $\hat\sigma_N$ and its theoretical value from the asymptotic normality, are displayed in Figure~\ref{fig:Ex1AdditiveFig3} (left panel). Similar to Example~1, the sample mean of the estimates converges to the true value, as the number of the Fourier modes $N$ increases, and the sample standard deviation of the estimates converges to zero at the rate predicted by the asymptotic normality property.
Finally, in Figure~\ref{fig:Ex1AdditiveFig3} (right panel), we present the empirical distribution of $\hat\sigma_N-\sigma$ for $N=55$, on which we superposed the distribution of Gaussian random variable (solid line) with mean zero and variance $\sigma^2/(2N)$, which validate the asymptotic normality of the estimators.
Various other model parameterizations consistently yield similar results, and the obtained numerical results agree with the theoretical results on consistency and asymptotic normality of $\hat\sigma_N$ from Theorem~\ref{th:AdditiveMLESigma}.

\begin{figure}[!ht]
    \centering
    \begin{subfigure}[t]{0.42\textwidth}
        \centering
        \includegraphics[width=\linewidth]{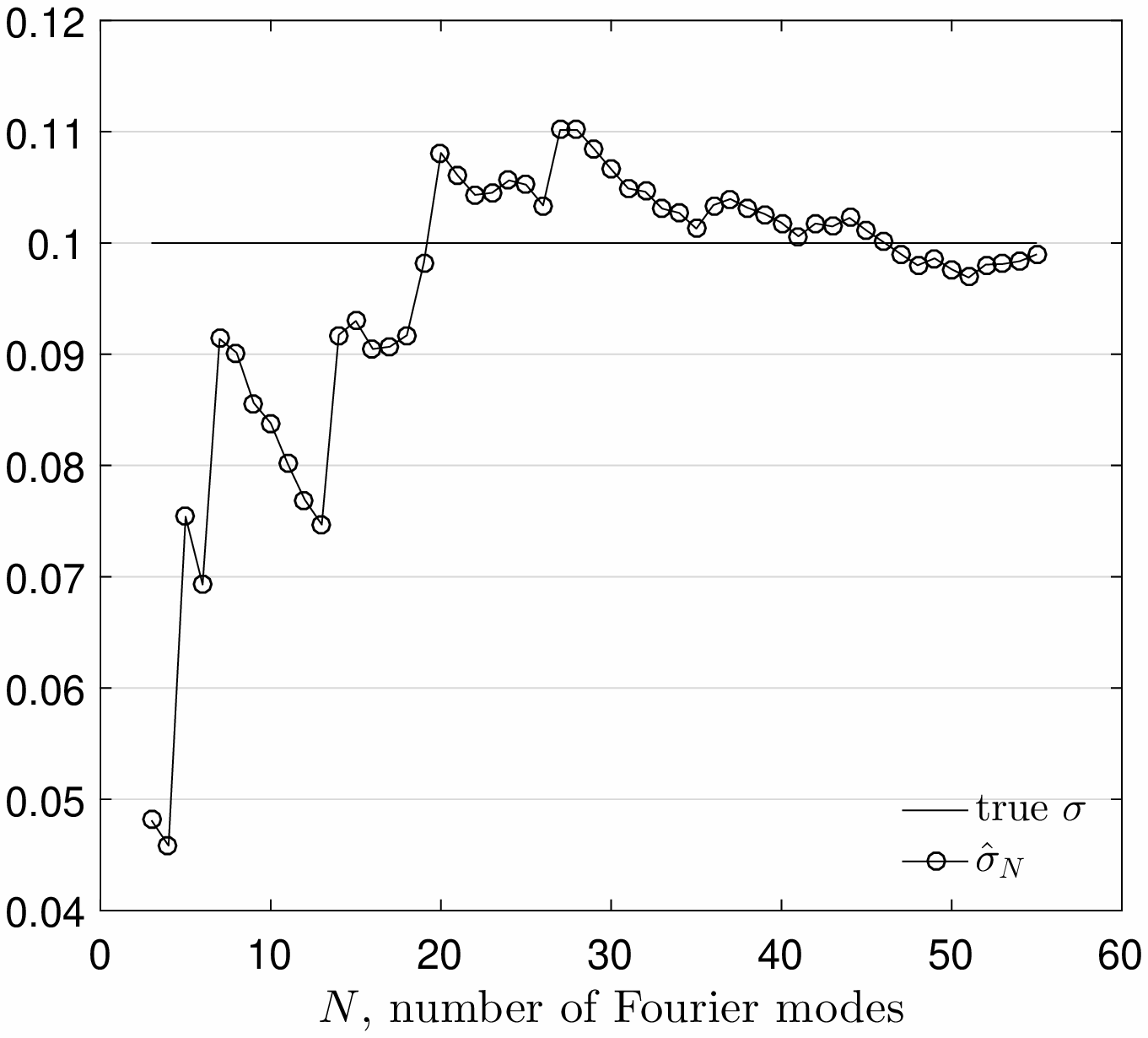}
    \end{subfigure}
    \hfill
    \begin{subfigure}[t]{0.48\textwidth}
        \centering
        \includegraphics[width=\linewidth]{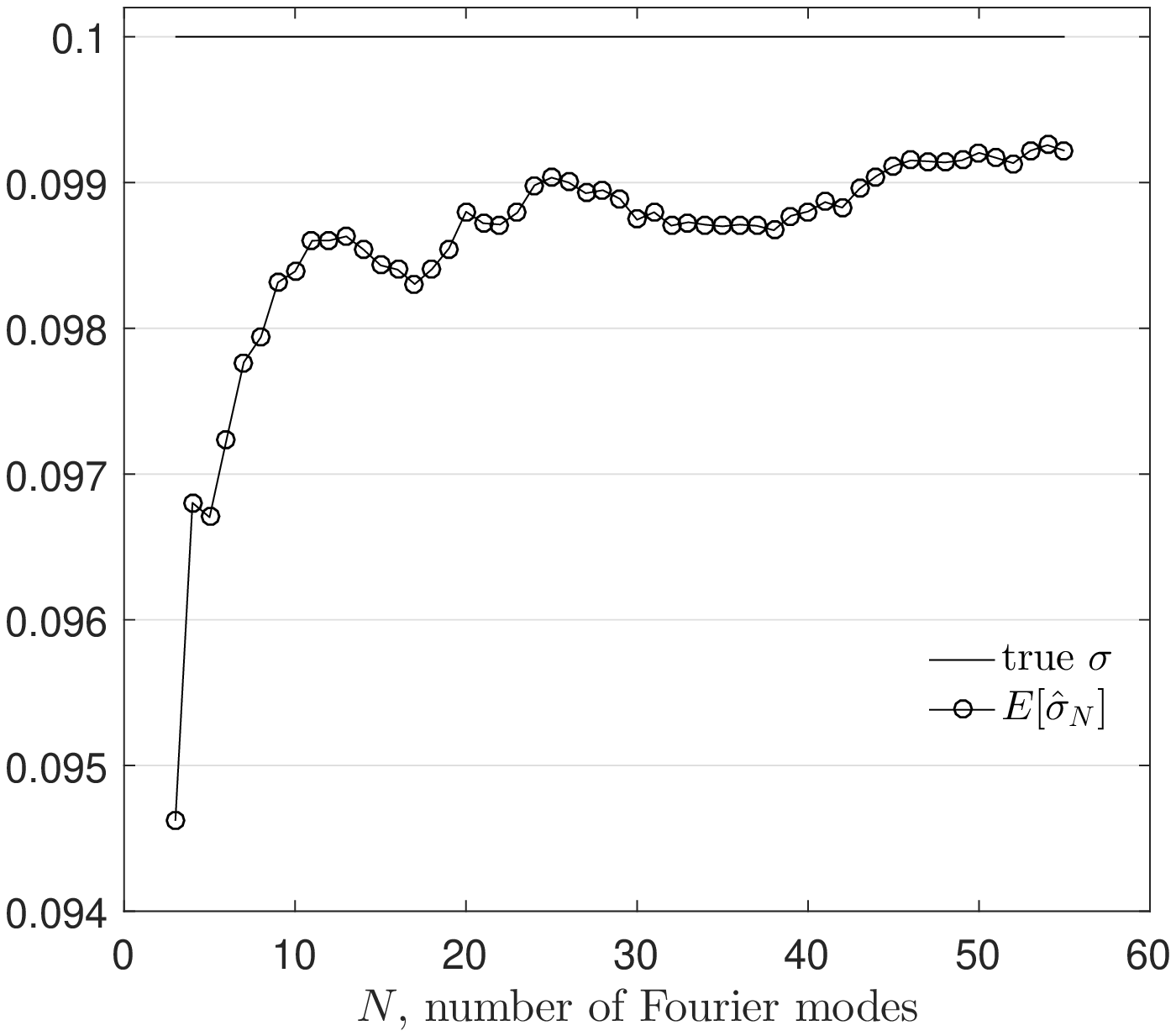}
    \end{subfigure}
    \caption{Additive noise. One sample path of $\hat\sigma_N$ (left panel) and the sample mean of $\hat\sigma_N$ (right panel).}
    \label{fig:Ex1AdditiveFig2}
\end{figure}

\begin{figure}[!ht]
    \centering
    \begin{subfigure}[t]{0.43\textwidth}
        \centering
        \includegraphics[width=\linewidth]{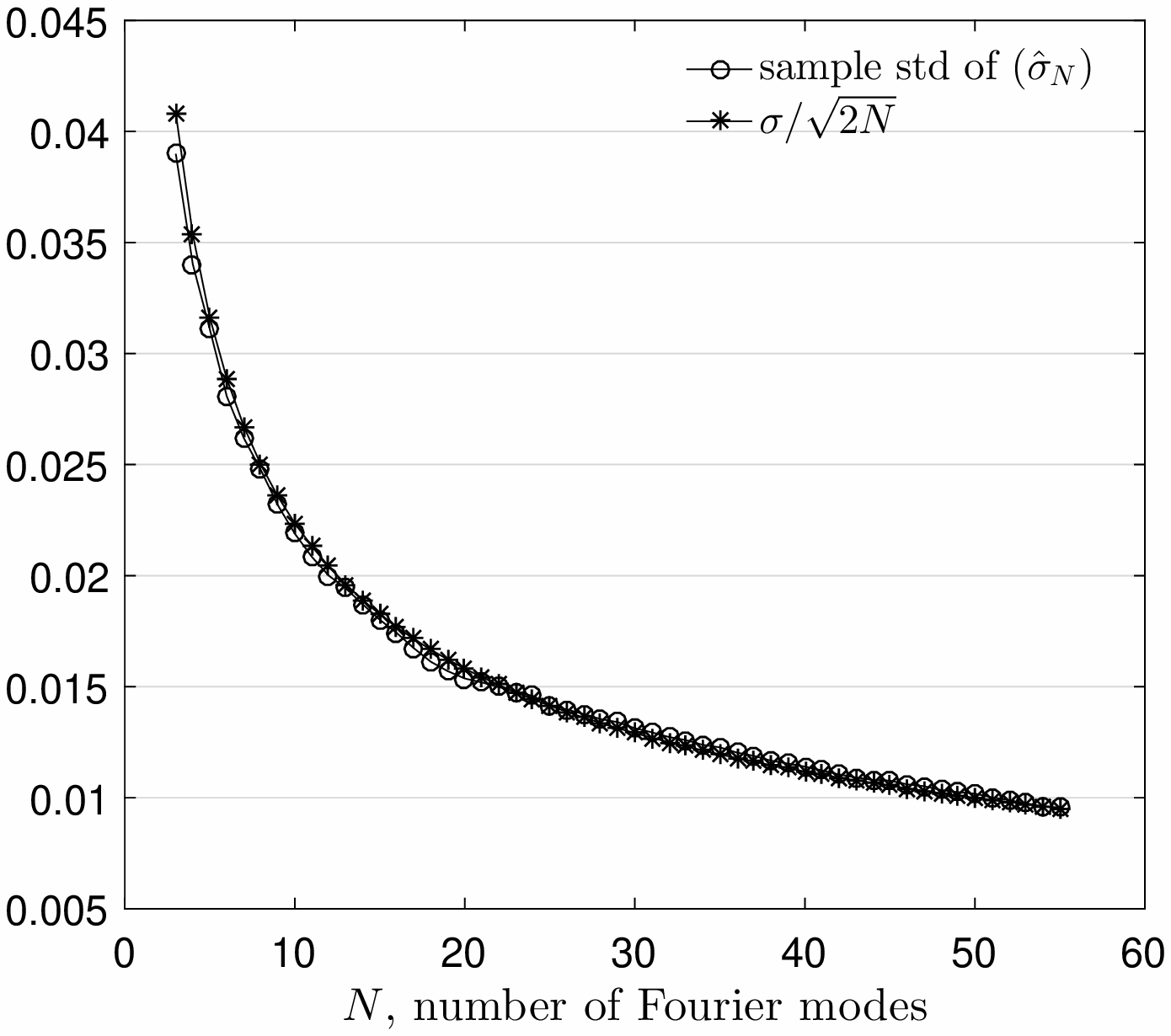}
    \end{subfigure}
    \hfill
    \begin{subfigure}[t]{0.48\textwidth}
        \centering
        \includegraphics[width=\linewidth]{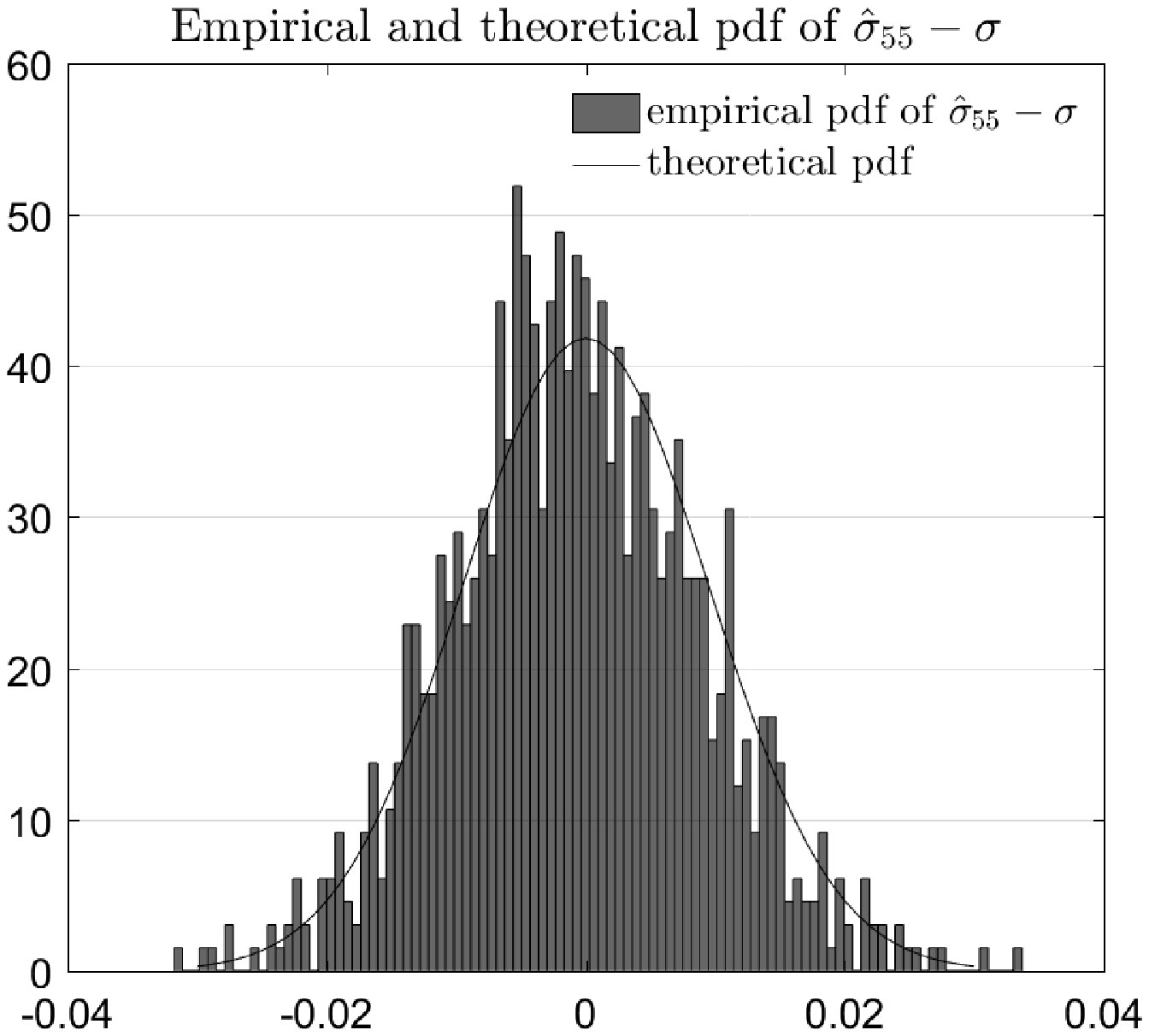}
    \end{subfigure}
    \caption{Additive noise. Sample standard deviation of $\hat\sigma_N$ and theoretical standard deviation $\sigma\sqrt{1/2N}$ from asymptotic normality (left panel) and  empirical distribution of $\hat\sigma_{55}-\theta$ and the pdf (solid lines) of the theoretical normal distribution from asymptotic normality (right panel).}
    \label{fig:Ex1AdditiveFig3}
\end{figure}

We conclude this example by applying the results from Section~\ref{sec:disc-samp-add}, assuming that $\sigma$ is known and $\theta$ is the parameter of interest. We postulate that the observer takes measurements of the spacial derivative of the solution $u_x$ at a fixed time point $t=0.2$ and over a space interval $[0,x]\subset[0,\pi]$. We approximate the function $u_x$ using the series representation \eqref{eq:ux}, and by taking the first $30,\!000$ terms in the series and a space resolution of $0.0015$. In Figure~\ref{fig:Ex1AdditiveFig4} (left panel)  we display the estimates $\check\theta_M$, and $\check\sigma_M$, given by \eqref{eq:thetaEST}, and \eqref{eq:sigmaEST}, and using the values of $u_x(t,\cdot)$ from interval $[0,x]$, for a set of values of $x\in(0,\pi]$. The obtained values are close to the true values. We also applied a similar approach to study the estimators  \eqref{eq:estTheta2} and \eqref{eq:estSigma2}. However, the slow convergence rate of the Fourier series combined with low smoothness order of $u$ produce less desirable numerical results, see Figure~\ref{fig:Ex1AdditiveFig4} (right panel). Further investigations are needed, tentatively by employing more accurate numerical methods to approximate the solution.

\begin{figure}[!ht]
    \centering
    \begin{subfigure}[t]{0.48\textwidth}
        \centering
        \includegraphics[width=\linewidth]{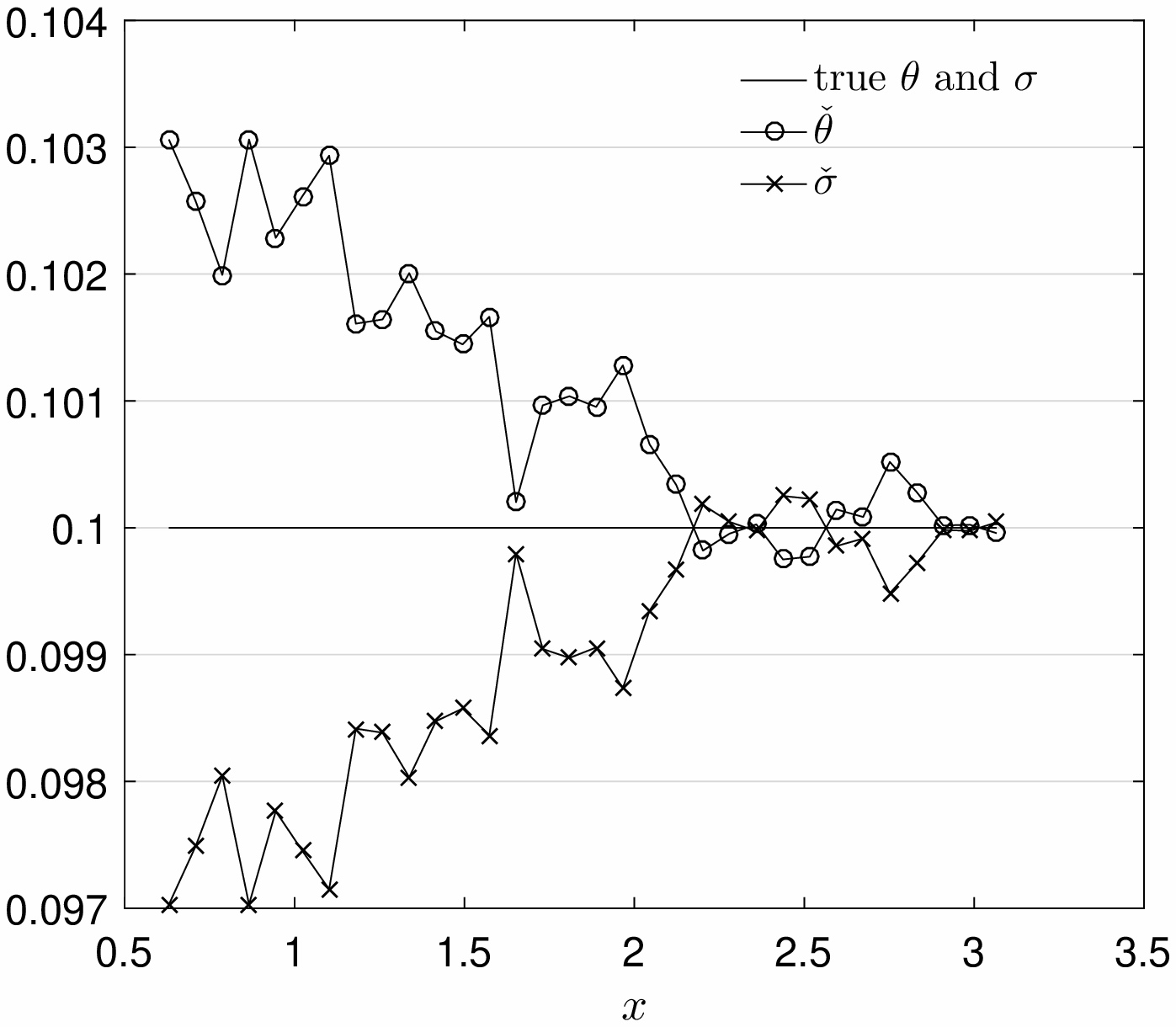}
    \end{subfigure}
    \hfill
    \begin{subfigure}[t]{0.48\textwidth}
        \centering
        \includegraphics[width=\linewidth]{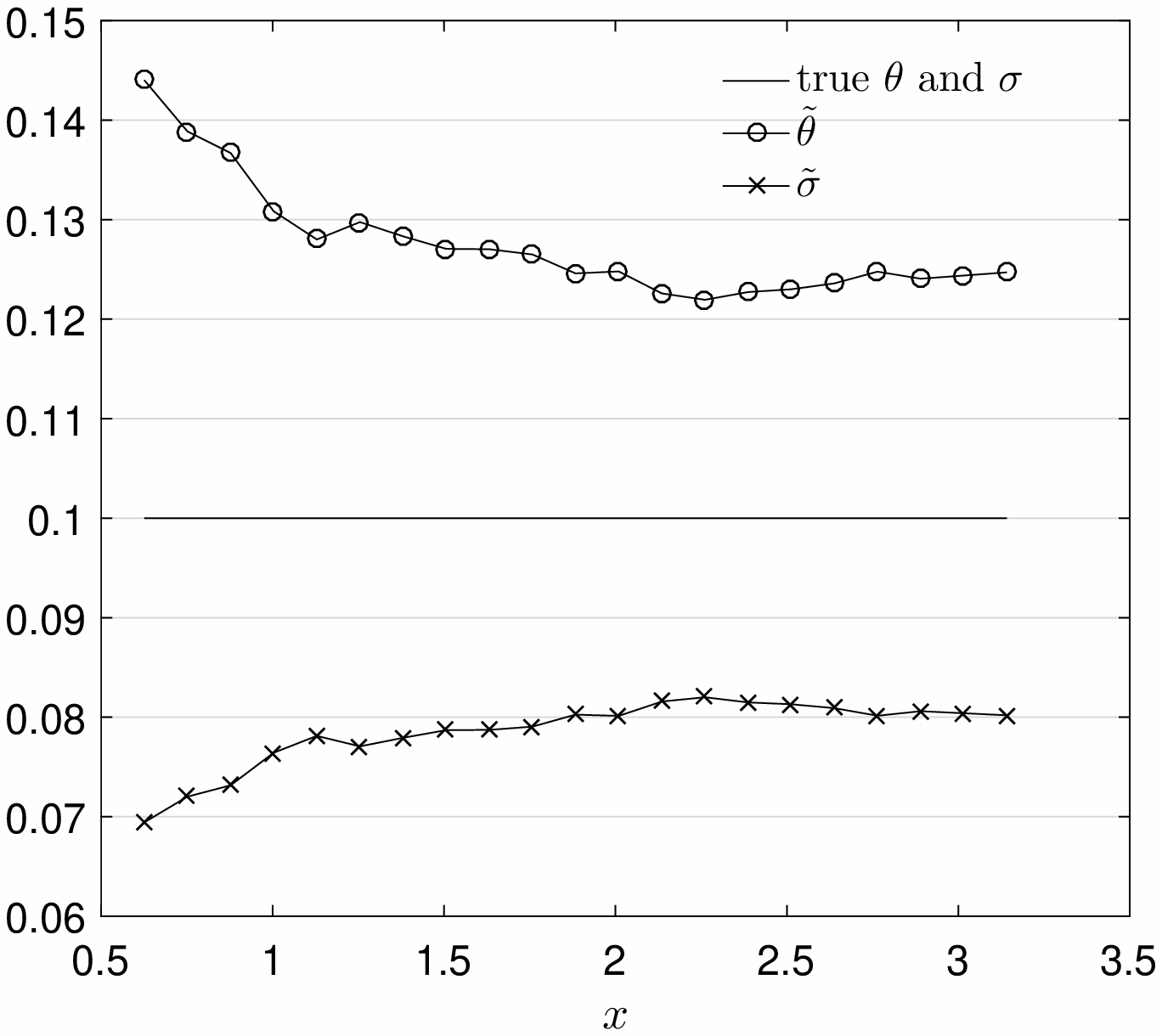}
    \end{subfigure}
    \caption{Additive noise. Discrete sampling. Left panel: values of $\check\theta$ and $\check\sigma$ by using values of $u_x(t,\cdot)$. Right panel:
     values of $\tilde\theta$ and $\tilde\sigma$ by using values of $u(t,\cdot)$. The horizontal axes $x$ indicates the right end point of $[0,x]$ over which the values of $u_x$ and $u$ were computed.}
    \label{fig:Ex1AdditiveFig4}
\end{figure}



\begin{thebibliography}{99}



\bibitem{ZCG2018}
Z.~Cheng, I.~Cialenco, and R.~Gong.
\newblock Bayesian estimations for diagonalizable bilinear {SPDEs}.
\newblock {\em Forthcoming in Stochastic Process. Appl.}, 2019. 

\bibitem{CialencoHuang2017}
I.~Cialenco and Y.~Huang.
\newblock A note on parameter estimation for discretely sampled {SPDE}s.
\newblock {\em Preprint arXiv:1710.01649}, 2017.

\bibitem{Cialenco2018}
I.~Cialenco.
\newblock Statistical inference for {SPDE}s: an overview.
\newblock {\em Statistical Inference for Stochastic Processes}, 21(2):309--329,
  2018.

\bibitem{CialencoLototsky2009}
I.~Cialenco and S.~V. Lototsky.
\newblock Parameter estimation in diagonalizable bilinear stochastic parabolic
  equations.
\newblock {\em Stat. Inference Stoch. Process.}, 12(3):203--219, 2009.

\bibitem{FriedlanderGlatt-HoltzVicol2016}
S.~Friedlander, N.~Glatt-Holtz, and V.~Vicol.
\newblock Inviscid limits for a stochastically forced shell model of turbulent
  flow.
\newblock {\em Ann. Inst. Henri Poincar\'{e} Probab. Stat.}, 52(3):1217--1247,
  2016.

\bibitem{Frisch1995}
U.~Frisch.
\newblock {\em Turbulence}.
\newblock Cambridge University Press, Cambridge, 1995.

\bibitem{Glatt-HoltzZiane2008}
N.~Glatt-Holtz and M.~Ziane.
\newblock The stochastic primitive equations in two space dimensions with
  multiplicative noise.
\newblock {\em Discrete Contin. Dyn. Syst. Ser. B}, 10(4):801--822, 2008.

\bibitem{HuebnerRozovskii1995}
M.~Huebner and B.~L. Rozovskii.
\newblock On asymptotic properties of maximum likelihood estimators for
  parabolic stochastic {PDE}'s.
\newblock {\em Probab. Theory Related Fields}, 103(2):143--163, 1995.

\bibitem{IbragimovKhasminskiiBook1981}
I.~A. Ibragimov and R.~Z. Has'minski\u{\i}.
\newblock {\em Statistical estimation: {A}symptotic theory}, volume~16 of {\em
  Applications of Mathematics}.
\newblock Springer-Verlag, New York-Berlin, 1981.

\bibitem{Kakutani1948}
S.~Kakutani.
\newblock On equivalence of infinite product measures.
\newblock {\em Ann. of Math. (2)}, 49:214--224, 1948.

\bibitem{KimLototsky2017}
H.-J. Kim and S.~V. Lototsky.
\newblock Time-homogeneous parabolic {W}ick-{A}nderson model in one space
  dimension: regularity of solution.
\newblock {\em Stoch. Partial Differ. Equ. Anal. Comput.}, 5(4):559--591, 2017.

\bibitem{KlebanerBook2005}
F.~C. Klebaner.
\newblock {\em Introduction to stochastic calculus with applications}.
\newblock Imperial College Press, London, second edition, 2005.

\bibitem{LototskyRozovsky2017Book}
S.~V. Lototsky and B.~L. Rozovsky.
\newblock {\em Stochastic partial differential equations}.
\newblock Universitext. Springer International Publishing, 2017.

\bibitem{RozovskyRozovsky2018Book}
S.~V. Lototsky and B.~L. Rozovsky.
\newblock {\em {S}tochastic {E}volution {S}ystems. Linear theory and
  applications to non-linear filtering}, volume~89 of {\em Probability Theory
  and Stochastic Modelling}.
\newblock Springer International Publishing, second edition, 2018.

\bibitem{NourdinPeccatiSwan2014}
I.~Nourdin, G.~Peccati, and Y.~Swan.
\newblock Entropy and the fourth moment phenomenon.
\newblock {\em J. Funct. Anal.}, 266(5):3170--3207, 2014.

\bibitem{ShiryaevBookProbability}
A.~N. Shiryaev.
\newblock {\em Probability}, volume~95 of {\em Graduate Texts in Mathematics}.
\newblock Springer-Verlag, New York, second edition, 1996.

\bibitem{Shubin}
M.~A. Shubin.
\newblock {\em Pseudodifferential operators and spectral theory}.
\newblock Springer-Verlag, Berlin, second edition, 2001.


\end{thebibliography}
\end{document}